\documentclass[12pt]{article}
\usepackage{amsfonts}
\usepackage{amsthm}
\usepackage{CJK}
\usepackage{amssymb}
\usepackage{amsmath}    
\usepackage{graphicx,floatrow}   
\usepackage{verbatim}   
\usepackage{color}      
\usepackage{subfigure}  
\usepackage{epsfig}
\usepackage{float}
\usepackage{epstopdf}
\usepackage{cases}
\usepackage{appendix}
\usepackage{bm}

\allowdisplaybreaks

\newtheorem{theorem}{Theorem}[section]
\newtheorem{lemma}[theorem]{Lemma}

\theoremstyle{definition}
\newtheorem{remark}{Remark}[section]
\newtheorem{example}[remark]{Example}
\numberwithin{equation}{section}

\newcommand{\norm}[1]{\left\Vert#1\right\Vert}
\newcommand{\abs}[1]{\left\vert#1\right\vert}

\begin{document}
\title{Two-phase image segmentation by the Allen-Cahn equation and a nonlocal edge detection operator}
\date{}
\author{Zhonghua Qiao \thanks{ Department of Applied Mathematics, The Hong Kong Polytechnic University, Hong Kong,  E-mail: zhonghua.qiao@polyu.edu.hk.}
\and Qian Zhang\thanks{Corresponding author. Department of Applied Mathematics, The Hong Kong Polytechnic University, Hong Kong, E-mail: qian77.zhang@polyu.edu.hk.}
}

\maketitle

\begin{abstract}
Based on a nonlocal Laplacian operator, a novel edge detection method of the grayscale image is proposed in this paper. This operator utilizes the information of neighbor pixels for a given pixel to obtain effective and delicate edge detection. The nonlocal edge detection method is used as an initialization for solving the Allen-Cahn equation to achieve two-phase segmentation of the grayscale image. Efficient exponential time differencing (ETD) solvers are employed in the time integration, and finite difference method is adopted in space discretization. The maximum bound principle and energy stability of the proposed numerical schemes are proved. The capability of our segmentation method has been verified in numerical experiments for different types of grayscale images.
\end{abstract}

\textbf{Key Words}: image segmentation, Allen-Cahn equation, nonlocal edge detection operator, maximum principle, energy stability

\section{Introduction}

Image segmentation is the process of dividing an image domain into some disjoint regions according to a characterization of the image within or in-between the regions \cite{Mitiche2011_SBH}. The characterization can be color, shape, edge, texture, or any feature that can distinguish each region from others. Image segmentation is involved in many fields, such as remote sensing, medical image recognition, robotics, and visual field monitoring.

Among many segmentation methods, variational methods have attracted considerable attention. A typical variational method for image segmentation is based on minimizing an objective energy functional, for instance, the Mumford-Shah (MS) model \cite{Mumford1989_CPAM},
\begin{align}\label{MS_model}
E_{MS}(u,I) = \int_{\Omega\backslash\Gamma} \abs{\nabla u}^2 dx + \mu Length(\Gamma)  + \lambda\int_{\Omega} (u-I)^2 dx,
\end{align}
which is associated with the partition determined by original image $I$ and a union of closed edges $\Gamma$. The function $u$ in the MS model is a piecewise smooth approximation to $I$. $\mu$ and $\lambda$ are positive constants. Under many circumstances, the function $u$ can degenerate to piecewise constant \cite{Chan2001_TIP} and the following model can be achieved:
\begin{align}\label{functional: cv}
E_{CV}(\Gamma,C_1,C_2) = Per(\Gamma; \Omega) +  \lambda_1 \int_{\Gamma} (C_1-I)^2  dx + \lambda_2 \int_{\Omega \backslash \Gamma} (C_2-I)^2  dx,
\end{align}
which is called the two-phase piecewise constant MS model, also known as Chan-Vese model considered previously by Chan and Vese with level set formulation in \cite{Chan2001_TIP,Vese2002_IJCV}. Here, $Per(\Gamma; \Omega)$ denotes the perimeter of the closed curve between two regions $\Gamma$ and $\Omega \backslash \Gamma$, $\lambda_1$ and $\lambda_2$ are positive constants, and $C_1, C_2$ are the average intensities of two phases respectively, which change along with the image intensity. Generally, the objective energy functional contained two parts, including image and partition constraints.  Image constraints yield  functional terms called fitting terms, which measure how close an approximation fits the original image.  Partition constraints give rise to partition priors which usually describe a geometric aspect of the edges, such as their length and smoothness. These two constraints have a wide variety, based on the region \cite{Li2005CVPR, Licm2007, Zhang_2010Pr} or edge \cite{Caselles1995_IEEE,Kass1988_IJCV}.

Different approximations combined with different constraints result in various energy functionals. To solve these models for image segmentation, the level set method, phase-field method, and threshold dynamic method have been successfully used. Wherein, level set methods were proposed by Osher and Sethian in \cite{Osher1988JCP,Osher2004AMR}. The closed curves describing a segmentation region can split or merge flexibly to implement segmentation with complex structures. Subsequently, it has been found that the energy functional can be replaced by another form and derived the variational level set method \cite{Li2005CVPR,Mitiche2011_SBH}, which has advantages on numerical stability. Inspired by the variational level set method, various phase-field models have been derived for image segmentation through the $\Gamma$-convergence \cite{Ambrosio1990_CPAM} for the length term of partition constraints, see e.g., \cite{Bertozzi2016SIAM,Jung2007_siam,Li2011CMA,Huang2019_JSC} and references therein.  Merriman, Bence, and Osher(MBO) \cite{MBO1992} developed a threshold dynamic method for the motion of an interface driven by the mean curvature. As pointed in \cite{Esedoglu2015_CPAM}, the length term of partition constraints can be estimated by the convolution of a heat kernel. As a consequence, a kind of threshold dynamic methods \cite{Esedoglu2006_JCP,Wang_2019, Wang_2017} had been constructed and obtained effective segmentation results.

In this paper, we plan to adopt a phase-field approach of the Chan-Vese model \eqref{functional: cv} for the two-phase segmentation of grayscale images. An alternating minimization method will be used to update the phase variable $u$, the two-phase average intensities $C_1$ and $C_2$ iteratively. Two exponential time differencing (ETD) schemes will be proposed to solve the resulted Allen-Cahn equation for $u$. The discrete maximum bound principle and energy stability will be proved under certain conditions. Notably, there needs an initial value to solve the Allen-Cahn equation. In general, the initial value is determined by taking an threshold value of the original image $I$ \cite{Cai_Zeng,Jung2007_siam,Li2011CMA,Huang2019_JSC}, or selecting a part of the image \cite{Chan2001_TIP, Li2005CVPR, Licm2007, Zhang_2010Pr}. An improper selection of initial values sometimes causes unsatisfactory segmentation results. Therefore, the setup of initial value is also a significant step for image segmentation. In this work, we will develop a nonlocal edge detection method, and use the detected edge as the initial value for solving the Allen-Cahn equation. The nonlocal edge detection method derived by using the nonlocal Laplacian operator, which is comparable and even superior to the existing gradient-based edge detection method \cite{Cherri1989_AO, Castleman1998_PH,Torre1986IEEE}. With the nonlocal edge detection method, more detailed information and structure can be retained in the initial value so that a more delicate image segmentation can be achieved.

The rest of this paper is organized as follows. In Section \ref{sec2}, we first review some widely-used edge detection operators based on the first-order and second-order derivatives, respectively. Then, we introduce a novel nonlocal edge detection method. Section \ref{sec3} is devoted to illustrating an iteration algorithm for the image segmentation by solving an Allen-Cahn type phase-field model, where we use the result of nonlocal edge detection as the initial value to start the iteration process. Plenty of numerical experiments are given in Section \ref{sec:numerical} to demonstrate the effectiveness and efficiency of our proposed methods.  Section \ref{conclusion} concludes the paper.

\section{Nonlocal edge detection method}\label{sec2}
In this section, we first review some mature gradient-based edge detection methods, and then introduce the nonlocal Laplacian operator. Based on the nonlocal Laplacian operator, we develop a novel nonlocal edge detection method.

\subsection{Some gradient-based edge detection operators}\label{sec:review}
 When the image intensity of a grayscale image changes intensely, the corresponding gradient will emerge the local extremum,  and it also associates with the zero-cross point of Laplacian. Thus the local extremum of gradients and the zero-cross point of Laplacian would be two significant measurements for detecting edge.

 Some detection methods are based on the first-order derivative, e.g., Roberts, Prewitt and Sobel operators \cite{Cherri1989_AO}. They use threshold to estimate the local extremum, after obtaining the corresponding approximate derivatives. The difference of aforementioned three methods lies in different weights of neighbor pixels. Though the consideration of neighbor pixels depresses the irrelevant noise, the edge detected by the first-order derivative edge detection methods is wider than expected, which makes the detected edge less accurate. Besides, the choice of the threshold value is by running trials, which causes the nondeterminacy of the detected edge.

  The well-known Canny edge detector \cite{Canny1986_IEEE} is looking for local extremum of the gradient. After a Gaussian filter on the gradient, this method uses two thresholds to detect edges, including weak edges in the output if they are connected to strong edges. By using two thresholds, the Canny method is less likely than the other methods to be disturbed by noise, and more likely to detect true weak edges.

For edge detection methods based on the second-order derivative, the zero-cross point of the Laplacian operator will indicate the edge. Applying the Laplacian operator to the original image directly will cause the double-edge phenomenon. Another deficiency of the Laplacian operator is the sensitivity to the noise. To fix these problems,
Marr proposed the Laplacian of a Gaussian operator as a remedy \cite{Marr1980_PRSL}:
\begin{align*}
L = \nabla^2[G(x,y)*I] = [\nabla^2G(x,y)]*I, \quad G(x,y)=\frac{1}{2\pi \varsigma^2}\exp(-\frac{x^2+y^2}{2\varsigma^2}),
\end{align*}
where $\displaystyle\nabla^2G(x,y) = -\frac{1}{\pi\varsigma^4}[1-\frac{x^2+y^2}{2\varsigma^2}]\exp(-\frac{x^2+y^2}{2\varsigma^2})$. Bigger $\varsigma$ gives us more blur of the image. After the above convolution, the zero-cross point of $L$ indicates the desired edge.

\subsection{Nonlocal Laplacian operator}\label{sec:nonlocal}
The pixel structure of the image provides a natural two-dimensional discretization on a fixed rectangular mesh, therefore we give the elaboration of the 2D nonlocal operator in the subsequent description. For the sake of convenience, the boundary condition is assumed to be the homogeneous Neumann boundary
\begin{align*}
\frac{\partial I}{\partial \mathbf{x}} = 0
\end{align*}
on the boundary layer is defined as
\begin{align*}
\Omega_I = \left\{ \mathbf{x} \not\in \Omega|\ dist(\mathbf{x},\Omega)\leq \delta \right\}.
\end{align*}
Here, $\Omega$ is the area where the original image $I$ belongs to. Without special illustration, we consider the grayscale image $I$ as a 2D matrix whose element denotes the pixel in the corresponding region.  Considering that the coordinate of image is always an integer, the interaction radius  $\delta$ is also taken as an integer.

\subsubsection{General representation}
We consider the following nonlocal Laplacian operator \cite{Qiang2018_IMA,Li2019siam}
\begin{align*}
\mathcal{L}_{\delta}u(\mathbf{x}) = \frac{1}{2}\int_{B_{\delta}(\mathbf{0})}\rho_{\delta}(\norm{\mathbf{s}})(u(\mathbf{x}+\mathbf{s}) - 2u(\mathbf{x}) + u(\mathbf{x}-\mathbf{s}))d\mathbf{s}, \ \mathbf{x}\in\Omega,
\end{align*}
where ${B_{\delta}(\mathbf{0})}$ is the ball with radius $\delta$ in plane $\mathbb{R}^2$ centered at the origin, $\rho_\delta: [0, \delta] \rightarrow  R$ is a nonnegative kernel function, and $\norm{\cdot}$  stands for the usual Euclidean norm. $\rho_{\delta}$ is further assumed to satisfy
\begin{align*}
\int_{B_{\delta}(\mathbf{0})}\norm{\mathbf{s}}^2\rho_{\delta}(\norm{\mathbf{s}})d\mathbf{s} = 2d,
\end{align*}
where dimension $d = 2$.

Given the uniform square mesh as $\mathbf{x_i} = h\mathbf{i}$, $\mathbf{i}\in\mathbb{Z}^2$. At any node $\mathbf{x_i}\in \Omega$, we rewrite the nonlocal Laplacian operator as
\begin{align*}
\mathcal{L}_{\delta}u(\mathbf{x_i}) = \frac{1}{2}\int_{B_{\delta}(\mathbf{0})}\frac{u(\mathbf{x_i+s}) - 2u(\mathbf{x_i}) + u(\mathbf{x_i-s})}{\norm{\mathbf{s}}^2}{\norm{\mathbf{s}}_1} \frac{\norm{\mathbf{s}}^2}{\norm{\mathbf{s}}_1}\rho_{\delta}(\norm{\mathbf{s}})d\mathbf{s}.
\end{align*}
With a quadrature-based finite difference discretization \cite{Qiang2018_IMA}, the nonlocal Laplacian  operator becomes
\begin{align*}
\mathcal{L}_{\delta,h}u(\mathbf{x_i}) = \frac{1}{2}\int_{B_{\delta}(\mathbf{0})}\mathcal{I}_h\left(\frac{u(\mathbf{x_i+s}) - 2u(\mathbf{x_i}) + u(\mathbf{x_i-s})}{\norm{\mathbf{s}}^2}{\norm{\mathbf{s}}_1}\right) \frac{\norm{\mathbf{s}}^2}{\norm{\mathbf{s}}_1}\rho_{\delta}(\norm{\mathbf{s}})d\mathbf{s}.
\end{align*}
The operator $\mathcal{I}_h$ here denotes the piece-wise multi-linear interpolation with respect to $\mathbf{s}$, which can be represented as
\begin{align*}
\mathcal{I}_hv(\mathbf{s}) = \sum\limits_{\mathbf{s}_\mathbf{j}}v(\mathbf{s}_\mathbf{j})\phi_\mathbf{j}(\mathbf{s}).
\end{align*}
Here, the basis function satisfies
\begin{align*}
\phi_\mathbf{j}(\mathbf{s_i}) =  \left\{
\begin{array}{rc}
0  ,   &  \text{if}   \  \mathbf{i} \neq \mathbf{j},\\
1    , &     \text{if} \ \mathbf{i} = \mathbf{j}.
\end{array} \right.
\end{align*}
Hence, the discrete nonlocal Laplacian operator is formulated as follows:
\begin{align*}
\mathcal{L}_{\delta,h}u(\mathbf{x_i}) = \sum\limits_{\mathbf{0} \neq \mathbf{s}_j \in B_{\delta}(\mathbf{0}) } \frac{u(\mathbf{x_i+s_j}) - 2u(\mathbf{x_i}) + u(\mathbf{x_i-s_j})}{\norm{\mathbf{s_j}}^2}\norm{\mathbf{s_j}}_1\beta_{\delta}(\mathbf{s_j}), \ \mathbf{x_i} \in \Omega
\end{align*}
where the coefficient $\beta_{\delta}(\mathbf{s_j})$ is
\begin{align*}
\beta_{\delta}(\mathbf{s_j}) = \frac{1}{2}\int_{B_{\delta}(\mathbf{0})}\phi_\mathbf{j}(\mathbf{s})\frac{\norm{\mathbf{s}}^2}{\norm{\mathbf{s}}_1}\rho_{\delta}(\norm{\mathbf{s}})d\mathbf{s}.
\end{align*}
The operator $\mathcal{L}_{\delta,h}$ is self-adjoint and negative semi-definite.
\subsubsection{Implementation on edge detection}

Let $I_{i,j}$ be the pixel value of the given image at the mesh point $(i,j)\in \mathbb{Z}^2$ and $r = \delta$ is a positive integer. The formulation of the nonlocal Laplacian operator is simplified as
\begin{align*}
L_{\delta,h}I_{i,j} = \sum\limits_{p =0}^{r}\sum\limits_{q =0}^{r}c_{p,q}( I_{i+p,j+q} + I_{i-p,j+q}+ I_{i+p,j-q}+ I_{i-p,j-q} - 4I_{i,j})
\end{align*}
where $c_{0,0} = 0$ and
\begin{align*}
c_{p,q} = \frac{p+q}{(p^2+q^2)h}\int\int_{B_{\delta}^+(0)} \phi_{p,q}(x,y) \rho_{\delta}(\sqrt{x^2 + y^2}) \frac{x^2 + y^2}{x+y}dxdy.
\end{align*}
$\phi_{p,q}$ is the bilinear basis function located at the point $(p,q)$.  The coefficient $c_{p,q}$ is symmetric which can be precomputed to reduce the complexity. In the process of image segmentation, the mesh size $h$ is taken as 1. We summarize our proposed nonlocal Laplacian edge detection method in Algorithm 1 below.
\\

\fbox{
\parbox{0.9\textwidth}{

\textbf{Algorithm 1}

\textbf{Step 1}: Given $\delta$, we obtain the nonlocal Laplacian
\begin{equation*}
L_{\delta,h}I_{i,j} = \sum\limits_{p =0}^{r}\sum\limits_{q =0}^{r}c_{p,q}( I_{i+p,j+q} + I_{i-p,j+q}+ I_{i+p,j-q}+ I_{i-p,j-q} - 4I_{i,j}).
  \end{equation*}

\textbf{Step 2}:  Find the desired edge by choosing a suitable threshold value $\sigma$,
\begin{align*}
e(i,j) = \left\{
\begin{array}{rc}
1   ,   &  \text{if}   \  {L_{\delta,h}I_{i,j} \geq \sigma},\\
0    , &     \text{if} \   {L_{\delta,h}I_{i,j} < \sigma}.
\end{array} \right.
\end{align*}
}}
\\

In Algorithm 1, there are two parameters $\delta, \sigma$ that need to be adjusted. We make the interaction radius $3 \leq \delta\leq 8$, which guarantees that the detection considers the neighbors' information as much as possible, meanwhile does not cause low contrast by too large $\delta$. And the threshold value $\sigma $ is taken slightly larger than zero to determine the zero-cross points approximately and detect the desired edge. In general, these two parameters vary with specific images, we will discuss it in Section \ref{sec:numerical}.

To illustrate the algorithm, a synthetic image profile $I_1 \in [0,1]$ is used as an example. We assume the profile $I_1$ consists of a weak edge, a noise point, a jump edge, and a stair edge
\begin{equation*}
 I_1 = [\underbrace{ 5\ 5 \  4.5\ 4.5\ }_{\text{weak edge}} \underbrace{ 0\ 0\ 0\ 7\ 0\ 0\ 0\ 0\  }_{\text{noise point}}  \underbrace{ 2\ 5\ 2\ 2\ }_{\text{jump edge}} \underbrace{ 0\ 0\ 0\ 0\ 7\ 7\ 7\ 7\ }_{\text{stair edge}} ]/10
\end{equation*}
 correspondingly.

  The nonlocal Laplacian operator is evaluated with $\delta = 8$. In Fig. \ref{fig:comparison_local}, we can see that the nonlocal Laplacian also has zero-cross points around the points where the image intensity changes sharply. $\sigma$ is set to be 0.02 for this example. The intersection of $y = \sigma$ and the nonlocal Laplacian is the desired edge $e$. The nonlocal effect avoids the appearance of zero-cross points at the weak edge and inapparent noise. We can decide whether a weak edge shows or not, by adjusting the threshold value according to the change of the image intensity.
  The threshold method also reduces the complexity of finding the zero-cross points point-wisely. However, it is still sensitive to strong noise points. To get an accurate segmentation of a grayscale image, we will introduce a phase-field model in next section, which takes the detected edge as initial data.

\begin{figure}[t!] \centering
			\includegraphics[width=0.90\columnwidth]{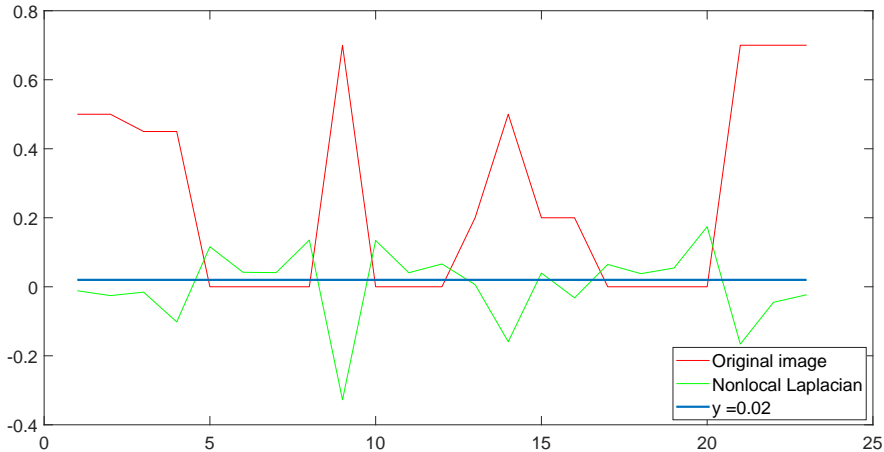}
		\caption{\label{fig:comparison_local} Nonlocal Laplacian edge detector for the profile $I_1$. }
	\end{figure}

\section{Image segmentation by the Allen-Cahn equation}\label{sec3}
The segmentation of a grayscale image from a simple edge detection is inaccurate and the process is very sensitive to noise. In this section, an Allen-Cahn model will be used to segment the given image to two parts with evident intensity change.

\subsection{A phase-field approach to the Chan-Vese model}
A phase-field approach to the Chan-Vese model \eqref{functional: cv} gives the following energy functional:
\begin{align}\label{eqn:AC_model}
E^\epsilon = \int_{\Omega}\left(\epsilon\abs{\nabla u}^2 + \frac{1}{\epsilon}W(u)+F(u)\right) dx,
\end{align}
where $$F(u)= \lambda_1[(C_1-I)^2 H(u-\frac{1}{2})] + \lambda_2[(C_2-I)^2(1-H(u-\frac{1}{2}))].$$
 Here, $W(u) = \sin^2{\pi u}$ is a periodic potential function \cite{Jung2007_siam},  the Heaviside function $H(u)$  is given by
 \begin{align*}
H(u) =  \left\{
\begin{array}{rc}
0  ,   &  \text{for}\ u < 0,\\
1    , &      \text{for}\ u \geq 0,
\end{array} \right.
\end{align*}
and $F(u)$ is the fitting term for the last two terms of \eqref{functional: cv}. The diffusion parameter $\epsilon$ and the weighted coefficients $\lambda_1, \lambda_2$ are positive. The phase variable $u$ eventually evolves to two phases 0 and 1. We assume the interface, i.e. the edge, lies at the contour of $\displaystyle u=\frac{1}{2}$.
\subsection{The numerical scheme}
The energy functional \eqref{eqn:AC_model} depends on the variables $u $ and $C_i, i = 1,2$. We adopt an alternating minimization scheme to minimize the energy functionals $E^\epsilon(u,\mathbf{C})$. Here $\mathbf{C}=(C_1, C_2)$. For the propose of illustration, we write the iteration process $ \mathbf{C}^k \rightarrow u^{k+1}\rightarrow \mathbf{C}^{k+1}$ as
\begin{align}
u^{k+1} = \text{argmin} \ E^\epsilon(\mathbf{C}^k,u^k), \label{eqn:iter1}\\
\mathbf{C}^{k+1} = \text{argmin} \ E^{\epsilon}(u^{k+1},\mathbf{C}^k). \label{eqn:iter2}
\end{align}
We give a regularization of the Heaviside function $H(u)$ by
\begin{align*}
H_{\epsilon_1}(u) = \begin{cases}&  \frac{1}{2\epsilon_1}\left(u + \frac{\epsilon_1}{\pi} \sin(\frac{\pi u}{ \epsilon_1})\right) + \frac{1}{2}, \ \  \text{for}\abs{u} \leq \epsilon_1, \\
& 1,  \ \  \text{for}\ u > \epsilon_1, \\
&0,  \ \  \text{for}\ u < -\epsilon_1.
\end{cases}
\end{align*}
Here, we take $\epsilon_1\leq \frac{1}{2}$. The derivative of $H_{\epsilon_1}(u)$ is an approximation of Dirac function $\displaystyle\delta(u)=\frac{d}{du}H(u)$ given by \cite{Li2005CVPR}
\begin{align*}
\delta_{\epsilon_1}(u)  = \begin{cases}& \frac{1}{2\epsilon_1}\left(1 + \cos(\frac{\pi u}{ \epsilon_1})\right), \ \  \text{for}\abs{u} \leq \epsilon_1, \\
                     &0,  \ \  \text{for}\ \abs{u} > \epsilon_1.
        \end{cases}
\end{align*}

For a fixed $u$, the solution for the equation \eqref{eqn:iter2} can be given by
\begin{align}\label{eqn:c}
C_1 = \frac{\int_{\Omega} H_{\epsilon_1}(u-\frac{1}{2})I dx}{\int_{\Omega} H_{\epsilon_1}(u-\frac{1}{2})dx} , \ \ C_2 = \frac{\int_{\Omega} \left(1-H_{\epsilon_1}(u-\frac{1}{2})\right)I dx}{\int_{\Omega}\left(1- H_{\epsilon_1}(u-\frac{1}{2})\right)dx}.
\end{align}
Therefore our attention should be paid on solving \eqref{eqn:iter1} by finding the steady state solution of the following Allen-Cahn equation:
\begin{align}\label{MAC_1}
u_t = 2\epsilon\Delta u - \frac{1}{\epsilon}w(u) - \left[\lambda_1(C_1-I)^2 - \lambda_2(C_2-I)^2\right]\delta_{\epsilon_1}(u-\frac{1}{2}),
\end{align}
where $w(u) = W'(u) = \pi \sin{2\pi u}$. We consider a homogenous Neumann boundary condition for this Allen-Cahn equation.

We summarize the algorithm described above as\\

\fbox{
\parbox{0.9\textwidth}{

\textbf{Algorithm 2}

\textbf{Step 0}: Give the initial values $u^0$ and ${C}^0_i, i = 1,2$, wherein $u^0$ by the nonlocal edge detection method, and $ C^0_1 = 1.0, C^0_2 = 0.0 $.

\textbf{Step 1}: Set a suitable $\epsilon$.  With  the initial value $u^{0}, {C}^0_i, i = 1,2$, solve the Allen-Cahn equation by the numerical scheme \eqref{ETD1} (or \eqref{ETD2}) till the steady state to get   an initial segmentation  $u^1$ of the original image. Update $C_1^1, C_2^1$ by \eqref{eqn:c} with $u^1$ and set a smaller $\epsilon$.

\textbf{Step 2}: With $u^k, C_1^k, C_2^k$, solve the Allen-Cahn equation by the numerical scheme \eqref{ETD1} (or \eqref{ETD2}) till the steady state to get $u^{k+1}$.

\textbf{Step 3}: Calculate the new ${C}_i^{k+1}, i = 1,2$  by \eqref{eqn:c} with  $u^{k+1}$.

\textbf{Step 4}: Set $u^{k}=u^{k+1}$. Repeat Steps 2-3 until the convergence of $u^{k+1}$.

}}\\

\begin{remark}
As mentioned in Section 2, the nonlocal edge detection is still sensitive to strong noises.
To remove the noise effect, we usually set a larger $\epsilon$ for Step 1 to get an initial segmentation of the original image.  This setup can eliminate the irrelevant broken lines or points detected by the nonlocal edge detection method, which underlines the primary part of the given image. We call this procedure as Stage 1 of the Algorithm 2. In Stage 2, we use a smaller $\epsilon$ in the remaining iterations to get a relatively sharp interface of two phases. As a consequence of Step 4, the boundary of the two parts is specified gradually as the process of iteration goes.
\end{remark}

\subsubsection{Exponential time differencing methods}
In this part, we propose two ETD methods to solve the Allen-Cahn equation \eqref{MAC_1}. Discretizing \eqref{MAC_1} in spatial variables by
the finite difference method leads to a system of ODEs:
\begin{align}\label{ODE}
U_t  + {L}_h U = {N}(U)
\end{align}
where
\begin{align*}
L_h = -2\epsilon D_h  + SI_d, \quad N(U) = SU - \frac{1}{\epsilon}w(U) - \left[\lambda_1(C_1-I)^2 - \lambda_2(C_2-I)^2\right]\delta_{\epsilon_1}(U-\frac{1}{2})
\end{align*}
with discrete Laplacian matrix $D_h$, which comes from the central finite different difference discretization of $\Delta$ with the homogeneous Neumann boundary condition, and identity matrix $I_d$. The positive constant $S$ is a stabilizer. Multiply \eqref{ODE} an integrating factor $e^{L_ht}$, and integrate from $t^n$ to $t^{n+1}$ gives
\begin{align}\label{eqn:ETD_integral}
U(t^{n+1}) = e^{-L_h\Delta t}U(t^n) + e^{-L_h\Delta t}\int_0^{\Delta t}e^{L_hs}N(U(t^n+s))ds.
\end{align}
Then we propose two approximations of \eqref{eqn:ETD_integral} as follows.

\begin{itemize}
\item ETD1 : Approximating $N(U(t^n+s))$ with  $N(U(t^n))$, we obtain
\begin{align}\label{ETD1}
U^{n+1} = \phi_0(L_h\Delta t) U^n +\Delta t \phi_1(L_h\Delta t) N(U^n).
\end{align}
\item ETDRK2: Evaluating $\hat{U}^{n+1}$ by ETD1 and approximating $N(U(t^n+s))$ with $\displaystyle(1-\frac{s}{\Delta t})N(U(t^n)) + \frac{s}{\Delta t}N(\hat{U}^{n+1})$, we obtain
    \begin{equation}\left\{
    \begin{split}
    &\hat{U}^{n+1} = \phi_0(L_h\Delta t) U^n +\Delta t \phi_1(L_h\Delta t) N(U^n),\\
    &{U}^{n+1} = \phi_0(L_h\Delta t) \hat{U}^{n+1} +\Delta t \phi_2(L_h\Delta t)( N(\hat{U}^{n+1})- N(U^n)). \label{ETD2}\\
         \end{split}\right.
     \end{equation}
\end{itemize}
$\phi_i,i = 0,1,2$ are defined as
\begin{align*}
\phi_0(a) = e^{-a}, \quad  \phi_1(a) = \frac{1-e^{-a}}{a}, \quad \phi_2 = \frac{e^{-a}-1+a}{a^2}.
\end{align*}

The key process of calculating $U^{n+1}$ from the scheme (\ref{ETD1})
or (\ref{ETD2}) is the efficient implementation of the actions of the exponentials $\phi_i(L_h\Delta t), i = 0,1,2$.
Since $D_h$ comes from the finite difference discretization of $\Delta$  with the homogeneous Neumann boundary condition, we could employ the 2D discrete cosine transform (DCT) in the calculation. We use a square domain to illustrate the implementation, where we have $N\times N$ pixels. We introduce $\mathcal{D}$ as the $2D$ DCT operator, and then for any $V = (V_{k,l})\in \mathbb{R}^{N\times N} $,
\begin{align*}
L_h = \mathcal{D}\hat{L}_h\mathcal{D}^{-1},
\end{align*}
where $$(\hat{L}_h V)_{k,l} = \lambda_{k,l}{V}_{k,l},\quad 1\leq k, l \leq N.$$
Here
\begin{align*}
\lambda_{k,l} = {8\epsilon}\sin^2{\frac{k\pi}{N}} + {8\epsilon}\sin^2{\frac{l\pi}{N}}+ S, \quad  1\leq k,l\leq N
\end{align*}
are the eigenvalues of $L_h$. Then we have \cite{Higham2008SIAM}
\begin{align*}
\phi_i(L_h\Delta t) =  \mathcal{D}\phi_i(\hat{L}_h\Delta t)\mathcal{D}^{-1}, \ (\phi_i(\hat{L}_h\Delta t)V)_{k,l} = \phi_i(\lambda_{k,l}\Delta t)V_{k,l}, \ i =0,1,2.
 \end{align*}
 The actions of $\mathcal{D}, \mathcal{D}^{-1}$ can be implemented by 2D DCT and its inverse transform, respectively.  The computational complexity is $\mathcal{O}(N^2\log N)$ per time step.

\subsubsection{Discrete maximum bound principle}
The solutions of the Allen-Cahn equation (\ref{MAC_1}) satisfy the maximum bound principle \cite{Li2021sirev}. Next, we will discuss the requirement for the stabilizer $S$ such that the solutions of ETD1 \eqref{ETD1} and ETDRK2 \eqref{ETD2} preserve the discrete maximum bound principle.

The initial value $u^0$ lies in the interval $[0,1]$. Taking the projection $\mathcal{P}: u \in [0,1] \rightarrow \tilde{u}\in [-1,1]$, we have
\begin{equation}\label{eqn:transformation}
\begin{split}
&\tilde{u} = \mathcal{P}(u) = 2(u - \frac{1}{2}), \ u = \mathcal{P}^{-1}(\tilde{u}) =  \frac{1}{2}\tilde{u} + \frac{1}{2}.
\end{split}
\end{equation}
Since we rescale the phase variable $u$, the original image $I$ and average intensities $C_1,C_2$ need to be changed accordingly. Using the same transformation \eqref{eqn:transformation}, we get the new $\tilde{I}$ lying in $ [-1,1]$ and the new average intensities $\tilde{C}_1, \tilde{C}_2$.
Substituting $\tilde{u}, \tilde{C}_1, \tilde{C}_2$ into the Allen-Cahn equation \eqref{MAC_1}, we have
\begin{align}\label{eqn:transformed}
\tilde{u}_t - 2\epsilon\Delta \tilde{u} + \frac{2}{\epsilon}\tilde{w}(\tilde{u}) + 2\left[\lambda_1(\tilde{C}_1-\tilde{I})^2 - \lambda_2(\tilde{C}_2-\tilde{I})^2\right]\delta_{\epsilon_1}(\frac{1}{2}\tilde{u})=0,
\end{align}
where $\tilde{w}(\tilde{u}) = \pi\sin{\pi(\tilde{u}+1)}$.   The ODE system obtained by the finite difference spatial discretization for equation \eqref{eqn:transformed} becomes
\begin{align}\label{ODE_trans}
\tilde{U}_t  + \tilde{L}_h \tilde{U} = \tilde{{N}}(\tilde{U}), 
\end{align}
with
\begin{align*}
 \tilde{L}_h = -2\epsilon D_h + \tilde{S}I_d,\ \ \tilde{N}(\tilde{U}) = \tilde{S}\tilde{U} - \frac{2}{\epsilon}\tilde{w}(\tilde{U}) - 2\left[\lambda_1(\tilde{C}_1-\tilde{I})^2 - \lambda_2(\tilde{C}_2-\tilde{I})^2\right]\delta_{\epsilon_1}(\frac{1}{2}\tilde{U}).
 \end{align*}

Subsequently, we prove  the solutions of ETD1 \eqref{ETD1} and ETDRK2 \eqref{ETD2} for the transformed equation \eqref{eqn:transformed} preserve the discrete maximum bound principle.
\begin{lemma}\label{lemma:bound}
The nonlinear term $\tilde{N}$ is bounded by $\tilde{S}$, i.e. $\norm{\tilde{N}(\cdot)}_{\infty} \leq \tilde{S}$, provided that
 \begin{align*}
 \tilde{S}\geq \tilde{L} \triangleq \frac{2\pi^2}{\epsilon} + \frac{2\lambda\pi}{\epsilon_1^2}, \quad \lambda \triangleq \max(\lambda_1,\lambda_2).
\end{align*}
\end{lemma}
\begin{proof}
The derivative of $\tilde{N(\xi)}$ is given by
\begin{align}\label{proof_b1}
\tilde{N}'(\xi) = \tilde{S} - \frac{2}{\epsilon}\tilde{w}'(\xi) - \left[\lambda_1(\tilde{C}_1-\tilde{I})^2 - \lambda_2(\tilde{C}_2-\tilde{I})^2\right]\delta'_{\epsilon_1}(\frac{1}{2}\xi).
\end{align}
For the second term of (\ref{proof_b1}), $-\frac{2}{\epsilon}\tilde{w}'(\xi) = -\frac{2\pi^2}{\epsilon}\cos{\pi(\xi+1)}$ is bounded is $\frac{2\pi^2}{\epsilon}$.
As for the third term, the coefficient of fitting term $\left[\lambda_1(\tilde{C}_1-\tilde{I})^2 - \lambda_2(\tilde{C}_2-\tilde{I})^2\right]$ can be estimated by
\begin{align*}
-4\lambda_2 \leq \left[\lambda_1(\tilde{C}_1-\tilde{I})^2 - \lambda_2(\tilde{C}_2-\tilde{I})^2\right] \leq 4\lambda_1
\end{align*}
i.e.
\begin{align*}
\abs {\lambda_1(\tilde{C}_1-\tilde{I})^2 - \lambda_2(\tilde{C}_2-\tilde{I})^2} \leq 4\max(\lambda_1,\lambda_2) \triangleq 4\lambda,
\end{align*}
since $\tilde{I}$ represents the rescaling image whose elements are in $[-1,1]$, and the average intensities $\tilde{C}_1, \tilde{C}_2$ are also bounded by $-1,1$. And we know $\delta'_{\epsilon_1}(\xi)=-\frac{\pi}{2\epsilon_1^2}\sin{\frac{\pi\xi}{\epsilon_1}}$.
Hence, the third term of $\tilde{N}'(\xi)$ is bounded by $\frac{2\lambda\pi}{\epsilon_1^2}$.
Let $$\tilde{L} \triangleq \frac{2\pi^2}{\epsilon} + \frac{2\lambda\pi}{\epsilon_1^2},$$ and then we have
\begin{align*}
 \tilde{N}'(\xi) \geq \tilde{S} - \tilde{L} \geq 0,
\end{align*}
when the stabilizer $\tilde{S}$ satisfies $\tilde{S} \geq \tilde{L}$.
On the other hand, the nonlinear term $\tilde{N}(\xi)$  also has property
\begin{align*}
\tilde{N}(-1) = -\tilde{S}, \quad \tilde{N}(1) = \tilde{S}.
\end{align*}
The monotonicity leads us to the conclusion,
\begin{align*}
\abs{\tilde{N}(\xi)} \leq \tilde{S},  \quad \xi\in[-1,1]
\end{align*}
under the condition $\tilde{S} \geq \tilde{L}$.
\end{proof}

With the bound of the nonlinear term $\tilde{N}$ in equation \eqref{ODE_trans}, the discrete maximum bound principle for the ETD1 \eqref{ETD1} and ETDRK2 \eqref{ETD2} can be derived by Theorem 3.4 and Theorem 3.5 in \cite{Li2019siam} when the stabilizer $\tilde{S} \geq \tilde{L} $. That is, the numerical solution $\tilde{U}^n$ satisfies $\norm{\tilde{U}^n}_{\infty} \leq 1 $ when the initial value $\tilde{u}^0$ is bounded by $-1$ and $1$. $\tilde{U}$ and $U$ also satisfy the relations in \eqref{eqn:transformation}, so numerical solutions of ETD1 \eqref{ETD1} and ETDRK2 \eqref{ETD2} for equation \eqref{ODE} preserve the discrete maximum bound principle, which is given by the following theorem.

\begin{theorem}\label{thm1}
 If the initial value satisfies $0\leq u^0\leq 1$, then numerical solutions $U^n$ of ETD1 \eqref{ETD1} and ETDRK2 \eqref{ETD2} for equation \eqref{ODE} is also in the interval $[0,1]$, for all $n>0$, provided that the stabilizer $S$ satisfies
\begin{align}\label{bound_L}
S\geq L,\ \ L = \tilde{L}.
\end{align}
\end{theorem}

Theorem \ref{thm1} demonstrates that for fixed $C_1, C_2$, our resolution for $u^k$ preserves the discrete maximum bound principle. It shows that the initial value of the next step in the iteration still lies in $[0,1]$.  After updating $C_1, C_2$ in Step 3 of Algorithm 2, the interval, which $C_1, C_2$ belong to, is invariant. It guarantees that the discrete maximum bound principle always holds for the whole iterative process.

\subsubsection{Discrete energy stability}
The Allen-Cahn equation (\ref{MAC_1}) holds the energy stability $$\frac{dE^{\epsilon}}{dt}\leq 0.$$ In this section, we show that the proposed ETD schemes could preserve this property in the discrete sense. For a rectangular image with $M\times N$ pixels, the solution of \eqref{ODE}, $U\in R^{MN}$. The discretized energy is defined as
$E_{h}^{\epsilon}$ defined by
$$E_h^{\epsilon}=\sum_{i=1}^{MN}(\frac{1}{\epsilon}W(U)+F_{\epsilon_1}(U))-{\epsilon}U^TD_hU, \quad \forall U\in R^{MN}.$$
Here, $F_{\epsilon_1}(U)=\lambda_1[(C_1-I)^2 H_{\epsilon_1}(U-\frac{1}{2})] + \lambda_2[(C_2-I)^2(1-H_{\epsilon_1}(U-\frac{1}{2}))].$

\begin{lemma}\label{lemma_1}
For the fixed constants $C_1,C_2$, the energy stability of the ETD schemes is given by
\begin{align*}
&\text{ ETD1 \eqref{ETD1}}: \quad E^{\epsilon}(U^{n+1})\leq E^{\epsilon}(U^n), \\
&\text{ ETDRK2 \eqref{ETD2}}: \quad E^{\epsilon}(U^{n+1})\leq E^{\epsilon}(U^n) + C(1+\Delta t)^2
\end{align*}
under the condition ${S}\geq \frac{L}{2}$. The constant $C$ is independent of $\Delta t$ and $h$.
\end{lemma}

\begin{remark}
Referring to the proofs of Theorem 5.1 and Theorems 5.2 in \cite{Li2019siam}, we can prove the discrete energy stability for the ETD1 and EDTRK2 schemes in a similar way. The essence of this proof lies in the estimate of  $\|F''_{\epsilon_1}\|_\infty$. From a similar argument for the bounds of the second term and third term of \eqref{proof_b1}, we can have $\|F''_{\epsilon_1}\|_\infty=L$.
\end{remark}

\begin{theorem}\label{thm2}
The iterative process \eqref{eqn:iter1}-\eqref{eqn:iter2} holds unconditional energy stability under the condition ${S}\geq \frac{L}{2}$,
\begin{align}\label{energy_stable}
&\text{ ETD1 }: E^{\epsilon}(u^{k+1},\mathbf{C}^{k+1}) \leq E^{\epsilon}(u^{k},\mathbf{C}^{k}),\\
&\text{ ETDRK2 }: E^{\epsilon}(u^{k+1},\mathbf{C}^{k+1}) \leq E^{\epsilon}(u^{k},\mathbf{C}^{k})+ \tilde{C},
\end{align}
where the constant $\tilde{C}$ is independent of $\Delta t$ and $h$.
\end{theorem}

\begin{proof}
We give the proof for the iterative process \eqref{eqn:iter1}-\eqref{eqn:iter2} with ETD1. The process with ETDRK2 is the same.

For the first stage \eqref{eqn:iter1}, Lemma \ref{lemma_1} derives
\begin{align*}
E^{\epsilon}(u^{k+1},\mathbf{C}^{k})\leq E^{\epsilon}(u^k,\mathbf{C}^{k}).
\end{align*}
And then the second stage \eqref{eqn:iter2} implies
\begin{align*}
E^{\epsilon}(u^{k+1},\mathbf{C}^{k+1})\leq E^{\epsilon}(u^{k+1},\mathbf{C}^{k}).
\end{align*}
According to the monotonicity of $E^{\epsilon}(u^{k+1},\mathbf{C}^{k+1})$ for $\mathbf{C}^{k+1}$, we obtain the results of Theorem \ref{thm2}.
\end{proof}

\begin{remark}
The energy stability in \eqref{energy_stable} for the ETDRK2 scheme cannot be obtained theoretically. However, numerical experiments in Section \ref{sec:numerical} will show that the energy stability \eqref{energy_stable} holds numerically for both ETD1 and ETDRK2.
\end{remark}

\section{Numerical experiments}\label{sec:numerical}
We divided this section into two parts. The first one is the edge detection by the nonlocal Laplacian operator. The other one is the image segmentation by the phase-field approach of the Chan-Vese model. The direct edge detection is sensitive to noise and may causes some broken edges, while the segmentation by the Allen-Cahn equation can fix these problems to handle more complex images. 

The fractional power kernel for the nonlocal Laplacian operator we choose is as follows,
 \begin{align*}
 \rho_{\delta}(r) = \frac{2(4-\alpha)}{\pi\delta^{4-\alpha}r^{\alpha}}\chi_{(0,\delta]}(r), \quad \alpha\in[0,4).
 \end{align*}
 In this section, the parameter $\alpha$ is taken as 1 for the integrality. Actually, there is no big difference when $\alpha$ changes. And the parameter in Stage 2 of Algorithm 2 is fixed as $\epsilon= 0.1$. The parameter $\epsilon_1$ involved with $H_{\epsilon_1},\delta_{\epsilon_1}$ is given by 0.5 for simplicity. The stabilizer $S$ is set according to the bound given in \eqref{bound_L}. All examples were performed on a Windows desktop system using an Intel Core i5 processor, and programmed in Matlab R2019a.

\subsection{Examples on edge detection}\label{NE_1}
In this subsection, we make the comparison between the nonlocal edge detection method in Section \ref{sec:nonlocal} and some gradient-based edge detection methods  in Section \ref{sec:review}. To quantified the error, some criteria are given to measure the image segmentation. Referring to \cite{Crandall2009_ECE, Liu2010PR}, the definitions of three kinds of ratios are listed as follows:
\begin{align*}
&\text{False positive ratio (FPR)} = \norm{S_1-S_2}_1/\norm{S_1}_1,\\
&\text{False negative ratio (FNR)}= \norm{S_1-S_2}_1/\norm{S_2}_1,\\
&\text{Ratio of segmentation error (RSE)} =  \norm{S_1-S_2}_1 /\left(\norm{S_1}_1 + \norm{S_2}_1\right)
\end{align*}
where $\norm{\cdot}_1$ is the sum of absolute value of each elements for the given matrix. $S_1$ is a synthetic exact image segmentation given by a two-phase image. And $S_2$ is a numerical segmentation result of the corresponding segmentation method.  
It is clearly seen that the smaller those criteria are, the more accurate the corresponding segmentation we have.
The nonlocal interaction radius $\delta$ changes slightly according to different images, and the threshold value $\sigma$ is taken as 0.05.

\begin{example}\label{example1}
We begin with a synthetic color image (size [1024,1024]) to test our nonlocal Laplacian operator. Because of the three channels of a color image, we only need one channel to extract the edge. Without losing generality,  we just exhibit the result of channel 1 in Table \ref{ex21tab} and Figure \ref{ex21}. It is shown that our nonlocal detection method can achieve smaller segmentation error than Roberts, Sobel and Log operators. Our nonlocal detection method is comparable to the Canny operator which is a pretty delicate gradient-based edge detection method. The nonlocal radius $\delta$ here is taken as 3. 

\begin{table}[!t]
\begin{tabular}{ccccccc}
\hline
                 & Roberts       & Sobel    & Log & Canny  &Nonlocal detection    \\ \hline
FPR   &   0.6703      &   0.7430   & 0.7989  &  0.6449 & 0.6952 \\
FNR   &   2.0328 & 2.8917& 2.5709& 1.8159& 1.6759\\
RSE   &   0.5041& 0.5911& 0.6095 & 0.4759& 0.4914\\\hline
CPU time& 5.67E-02& 4.66E-02 &1.53E-01 &1.04E-01& 1.81E-01 \\\hline
\end{tabular}
\caption{\label{ex21tab} Comparison of different edge detection methods in Example \ref{example1}.}
\end{table}

\begin{figure}[t!]\centering
 \vspace{-2cm}
{\subfigcapskip=-65pt
		 \subfigure[The original image]{
			\includegraphics[width=0.3\linewidth, height=0.6\linewidth]{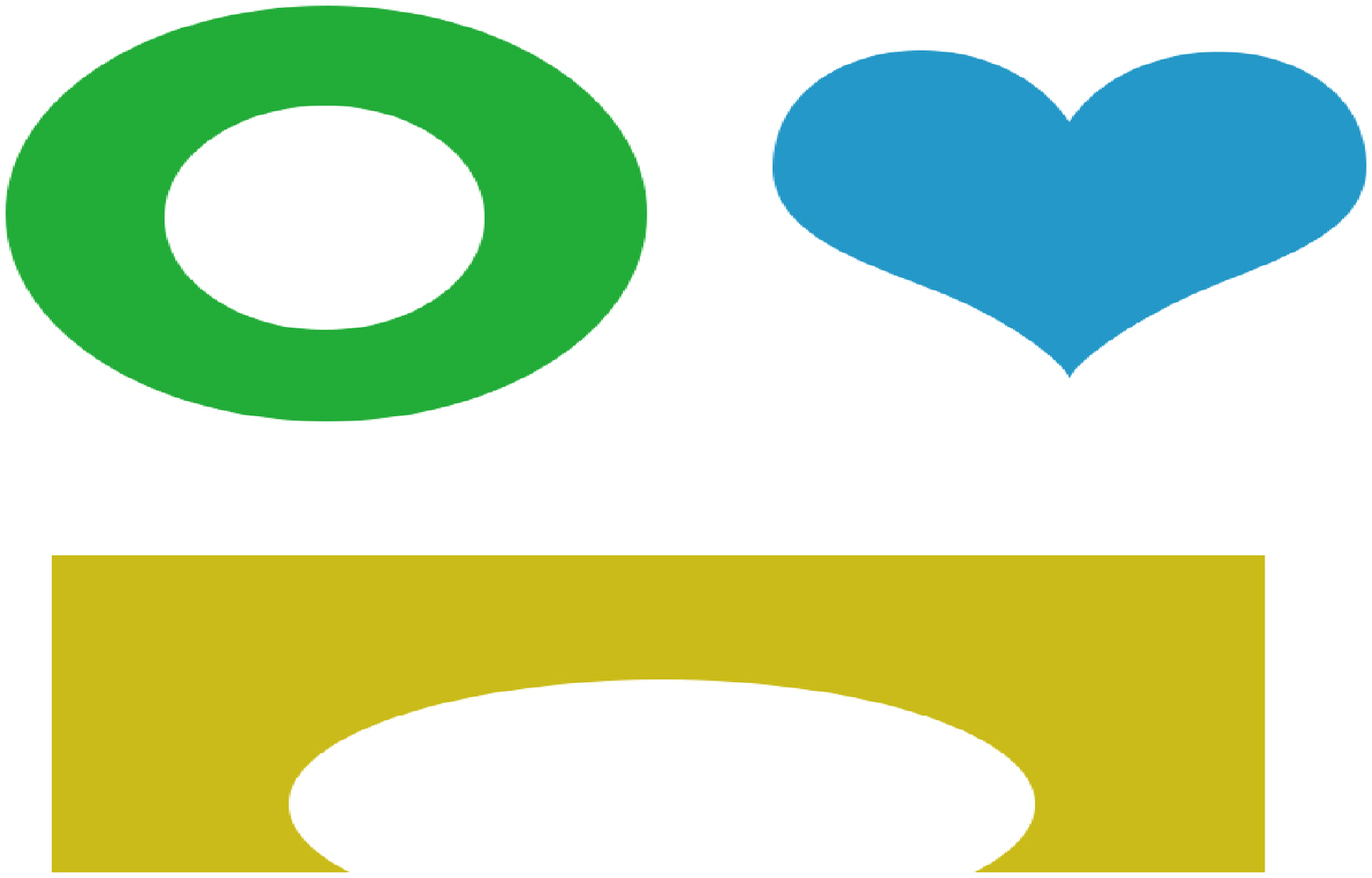}}
}
{
\subfigcapskip=-65pt
       \subfigure[The grayscale image]{
             \includegraphics[width=0.3\linewidth, height=0.6\linewidth]{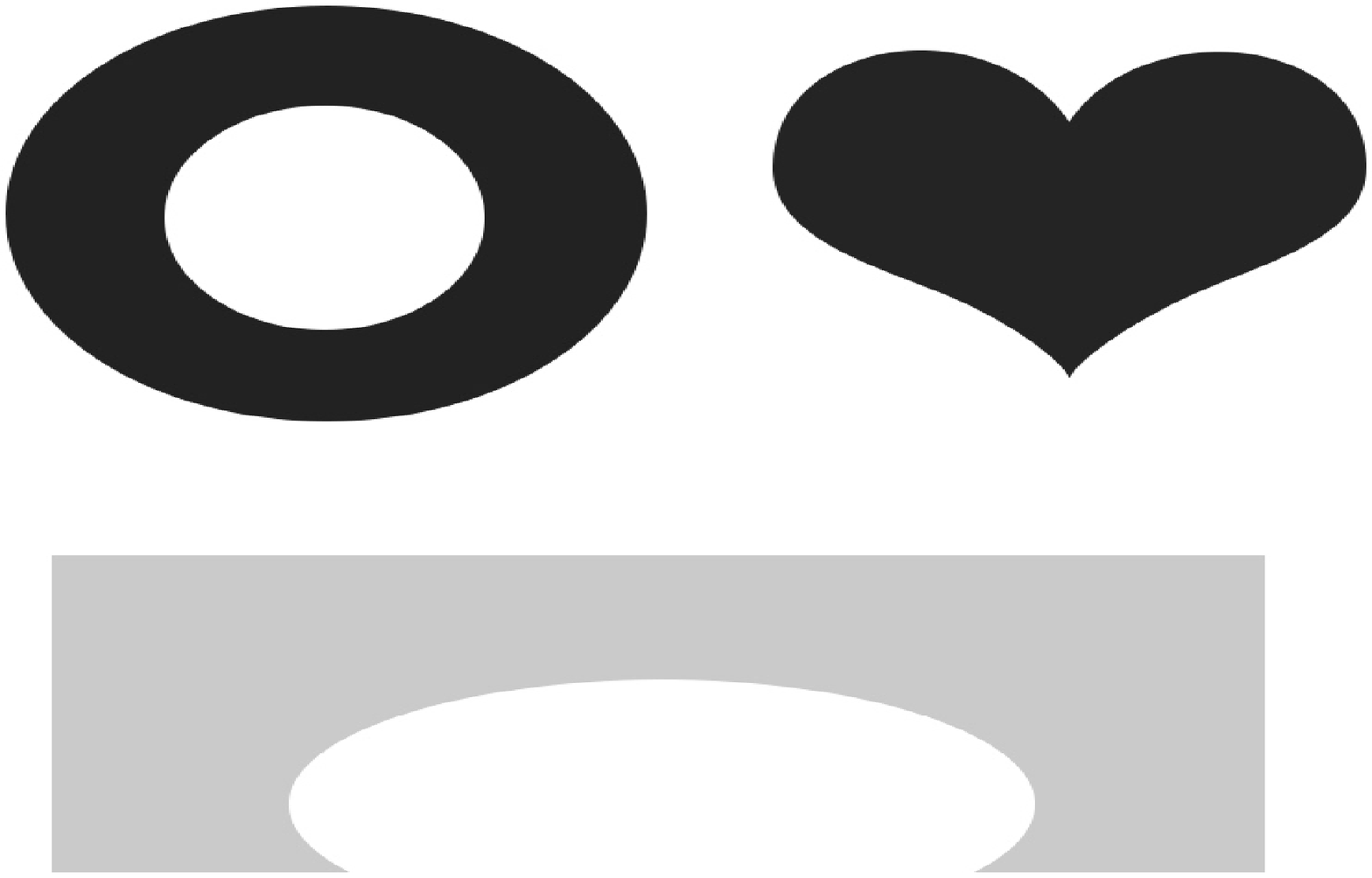}}\\
 }
 \vspace{-1cm}
         \subfigure[Edge detected by the nonlocal operator]{
			\includegraphics[width=0.3\linewidth, height=0.3\linewidth]{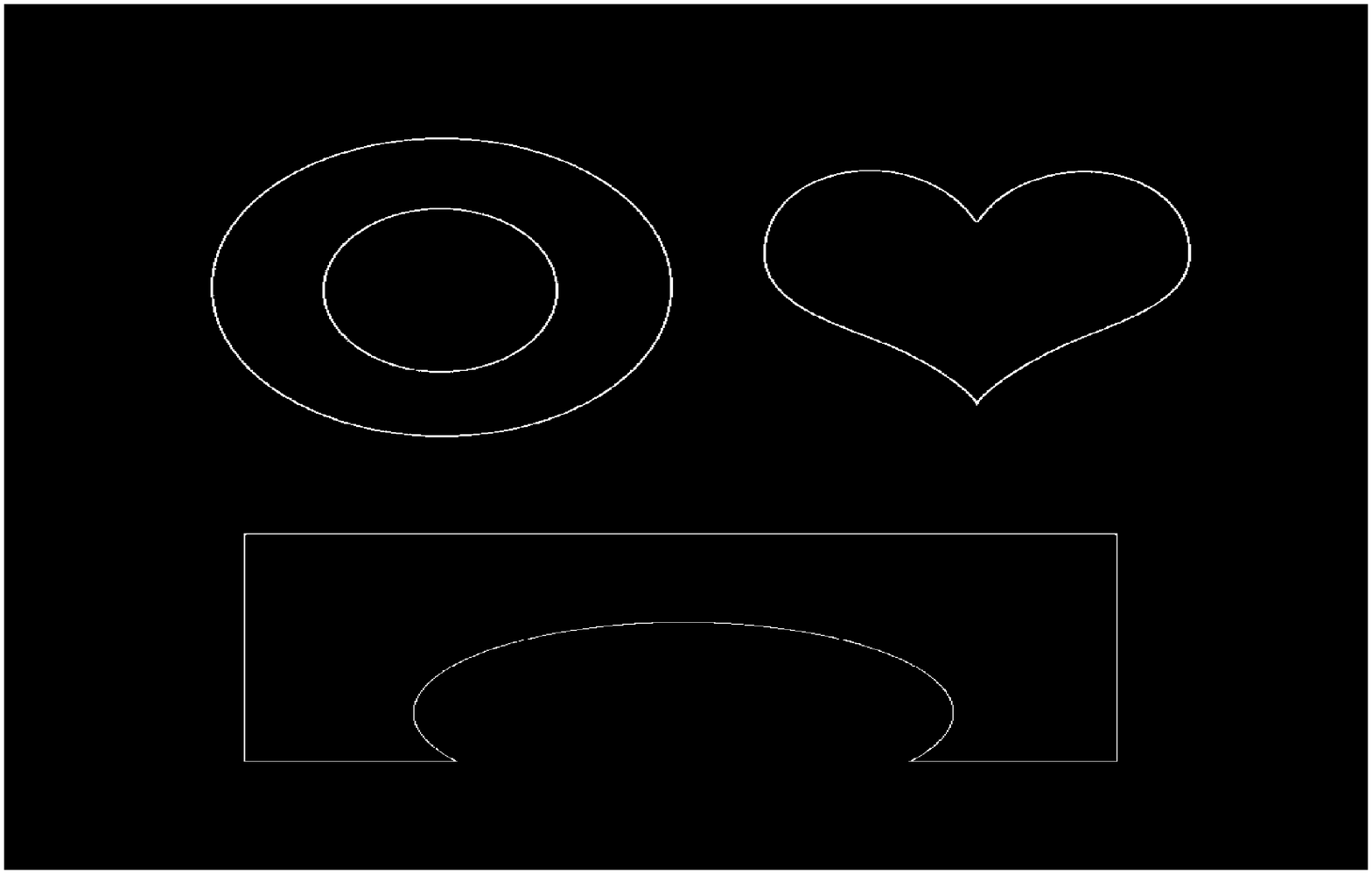}
}
\subfigure[The exact edge]{
			\includegraphics[width=0.3\linewidth, height=0.3\linewidth]{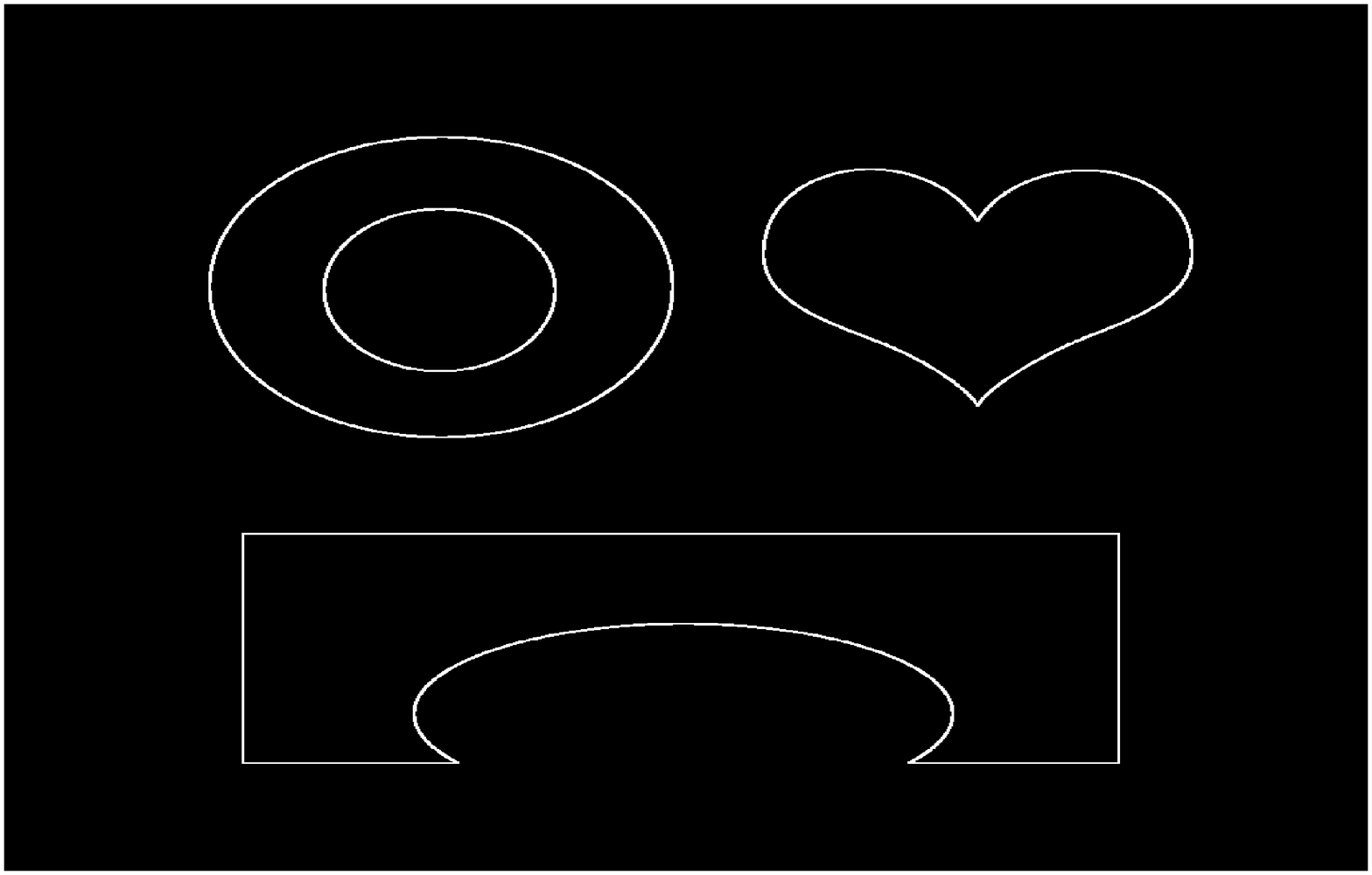}
		}
		
		\caption{\label{ex21} Nonlocal edge detection for channel 1 of a color image in Example \ref{example1}. }
	\end{figure}
\end{example}

\begin{example}\label{example2}
 We choose a realistic image `cameraman'(size [256,256]), which is used frequently to test the edge detection operator. In this example, the result segmented by the Canny operator  is taken as the reference. The segmentation in Figure \ref{figex5} shows that the right hand of the `cameraman' can be recognized by the Canny operator, where our nonlocal operator has some defects.  Whereas, the recognition of the building and the grassland makes our nonlocal edge detection method outperform other detection methods. Here $\delta$ is set to be 4. 

\begin{figure}[!t]\centering
        \subfigure[The original image]{
			\includegraphics[width=0.3\columnwidth]{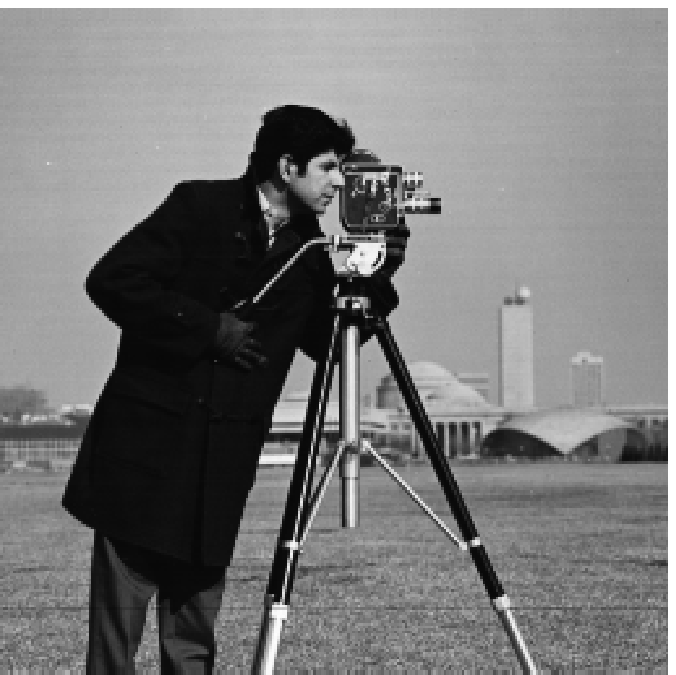}
		}
		\subfigure[The nonlocal detection]{
			\includegraphics[width=0.3\columnwidth]{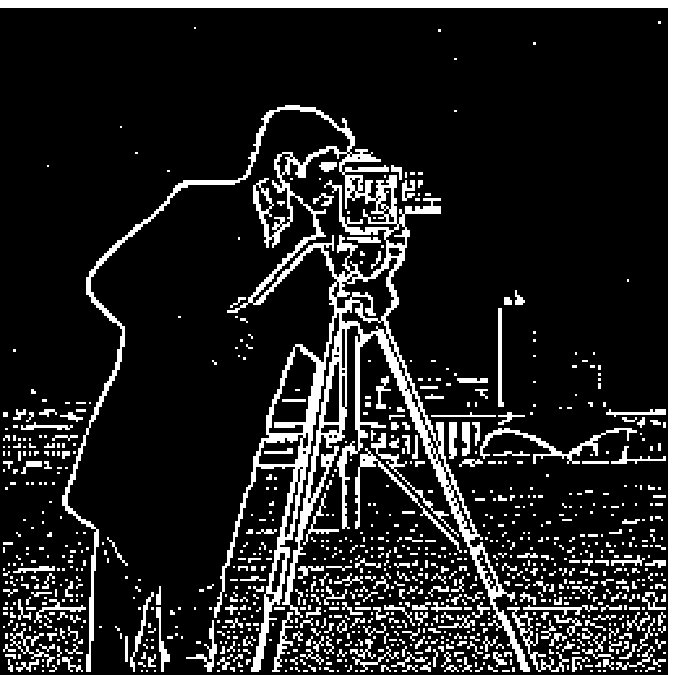}
		}
\subfigure[Roberts]{
			\includegraphics[width=0.3\columnwidth]{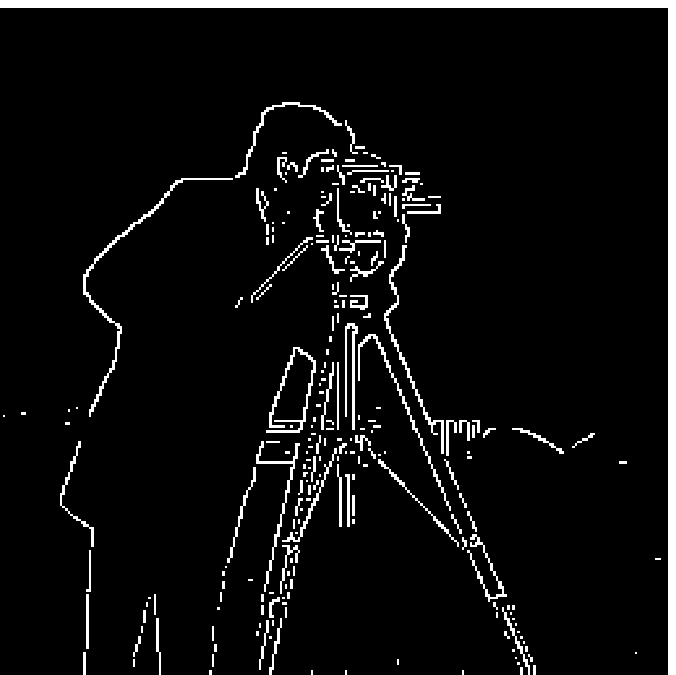}
		}
\subfigure[Sobel]{
			\includegraphics[width=0.3\columnwidth]{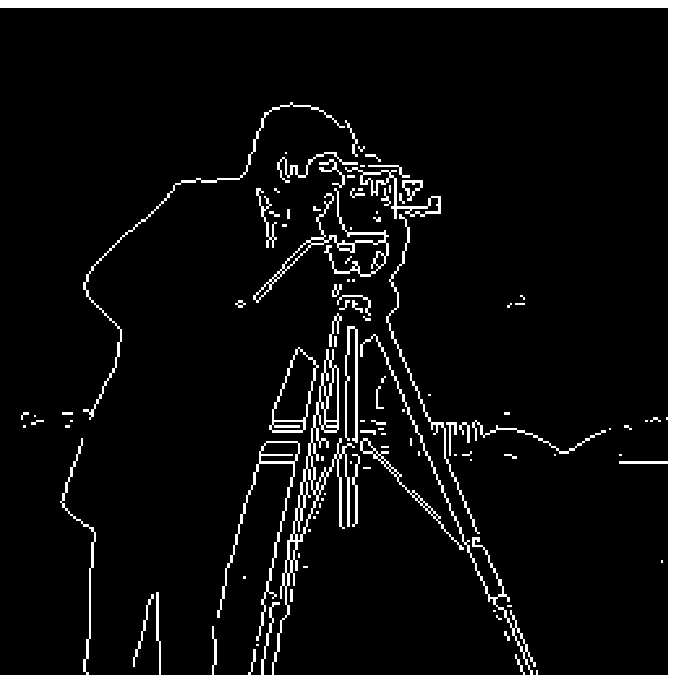}
		}
\subfigure[Log]{
			\includegraphics[width=0.3\columnwidth]{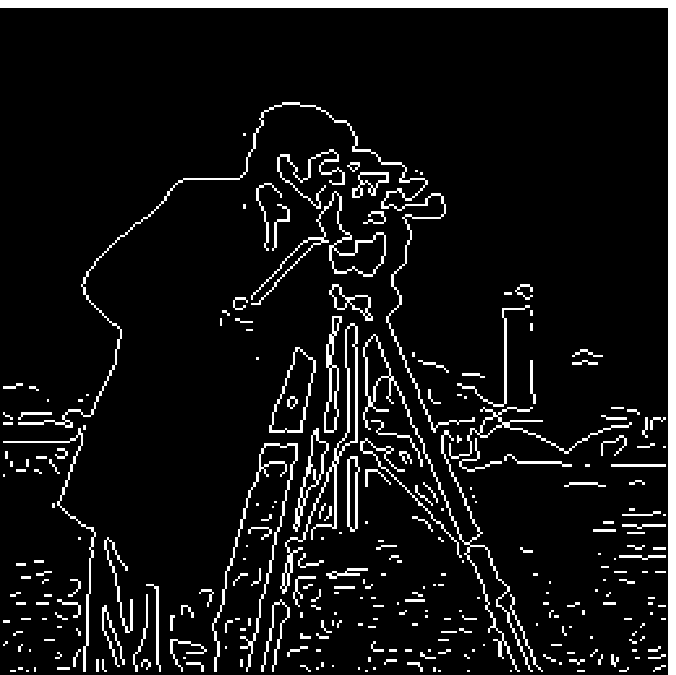}
		}
\subfigure[Canny]{
			\includegraphics[width=0.3\columnwidth]{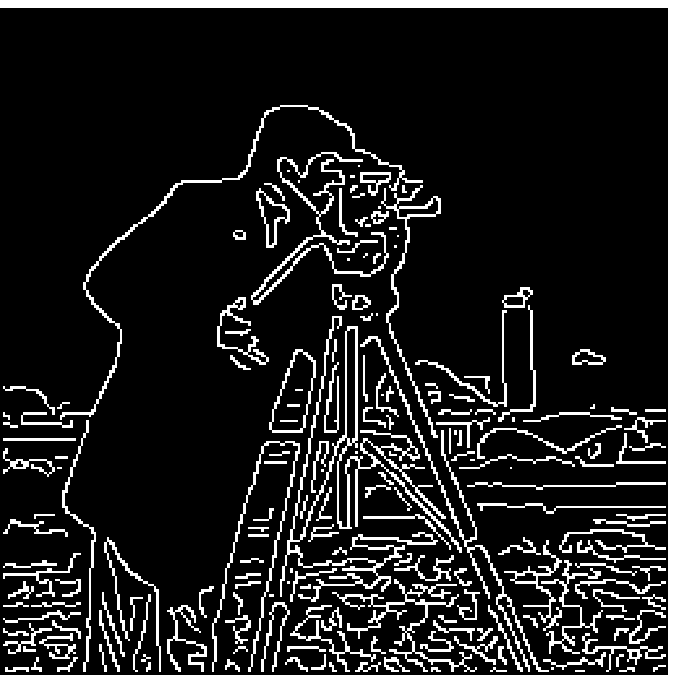}
		}
		\caption{\label{figex5}  Edge detection results for the image `cameraman' in Example \ref{example2}. }
	\end{figure}



\end{example}

\subsection{Examples on image segmentation by the Allen-Cahn equation}

When we solve the Allen-Cahn equation \eqref{MAC_1} to segment the given image, the whole iteration consists of two stages. An initial segmentation is generated in Stage 1 with larger $\epsilon$. Stage 2 with smaller $\epsilon$ specifies the desired edge. We denote $k_1$ as the steps to reach the steady state solution of the Allen-Cahn equation in Stage 1, and $k_i,i=2,\cdots,m$ for each steady state in Stage 2. The number of total loops for updating $C_1,C_2$ is $m$. In the simulation, we record the whole CPU time and verify the discrete maximum bound principle of our methods. 

\begin{example}\label{example_noise}
This example considers the segmentation of a synthetic image (size $[648,1152]$) polluted by Gaussian noise with zero mean and standard derivation $0.2$.  The image with such kind of noise can not be segmented by the nonlocal edge detection method directly. The phase-field approach of the Chan-Vese model is applied to segment this image in Figure \ref{fig:noise}. The red line is our detected edge plotted as the contour $u =\frac{1}{2}$. The corresponding parameters are listed as $\lambda_1 = 1.0 ; \lambda_2 = 1.0 $,  $\epsilon= 5$ for Stage 1, interaction radius $\delta = 5$ and threshold value $\sigma = 0.05$ for the initial nonlocal detection.

 The comparison of two initialization methods, the threshold method and nonlocal detection method, is contained in Table \ref{extab:noise} for first-order (ETD1) and second-order (ETDRK2) schemes, respectively. Compared with the first-order scheme, the second-order scheme can reduce $k_1, k_2$ validly, likewise the number of exterior loop $m$.   Because this image is syntectic, the exact segmentation is available. The error $E = \frac{\norm{U-I_{ex}}_1}{\norm{U}_1}$ between the exact segmentation $I_{ex}$ and the numerical result $U$ for all four cases are similar and quite accurate.  Although the time taken by the nonlocal-detection initialization method is a little longer than the threshold method, we will explain the superiority of the nonlocal-detection initialization method with some real images later. The maximum bound principle is confirmed in the last two column of Table \ref{extab:noise}. In summary, the nonlocal detection method is comparable to the threshold method to generate initial data for solving the Allen-Cahn equation.  The energy stability is shown numerically in Figure \ref{energy_plot}. It can be observed that both schemes are energy non-increasing, which implies that our numerical schemes are energy stable.

 \begin{figure}[!t]\centering
        \subfigure[The original image]{
			\includegraphics[width=0.45\linewidth, height=0.15\linewidth]{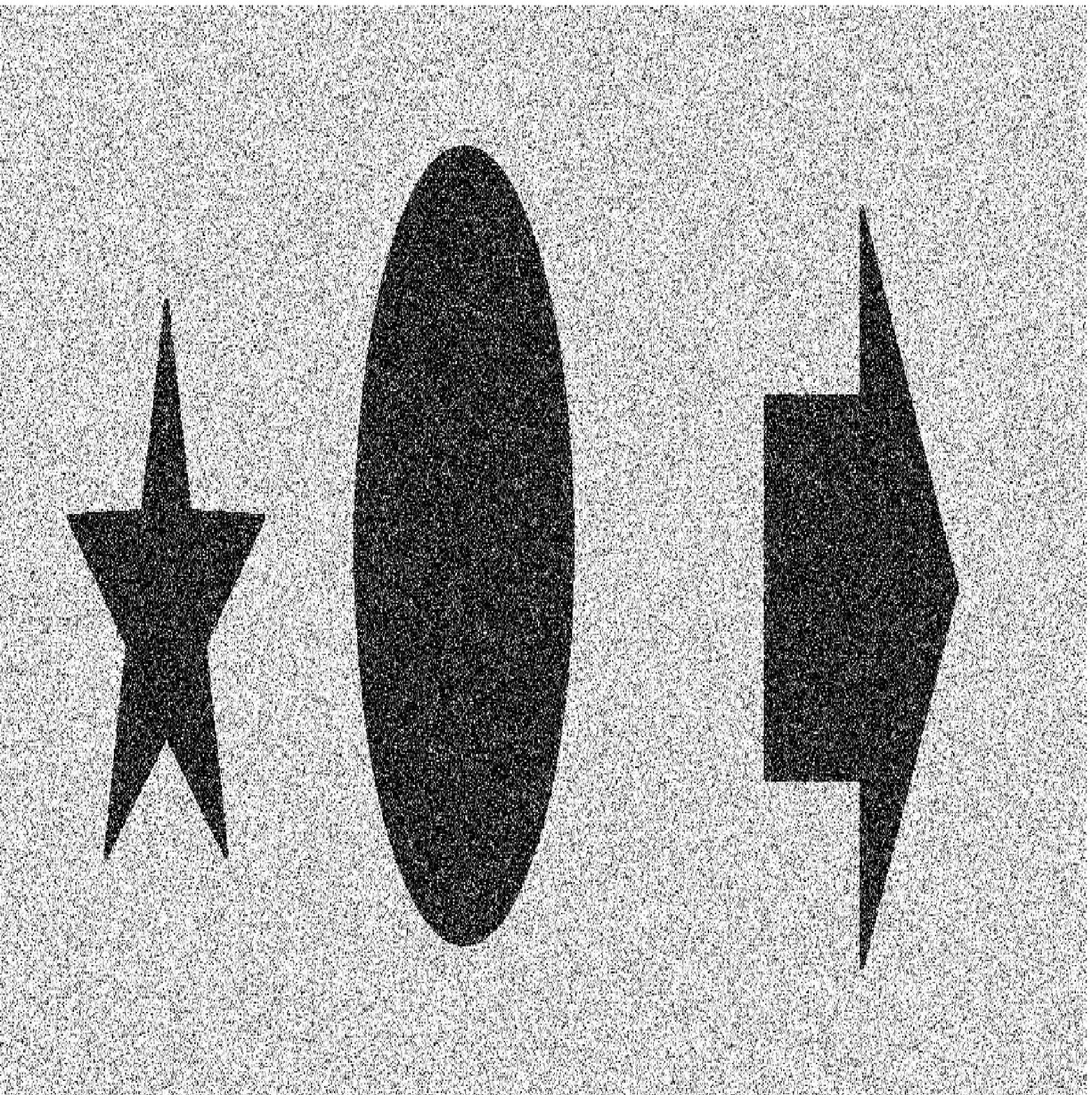}
		}
		\subfigure[The nonlocal edge detection]{
			\includegraphics[width=0.45\linewidth, height=0.15\linewidth]{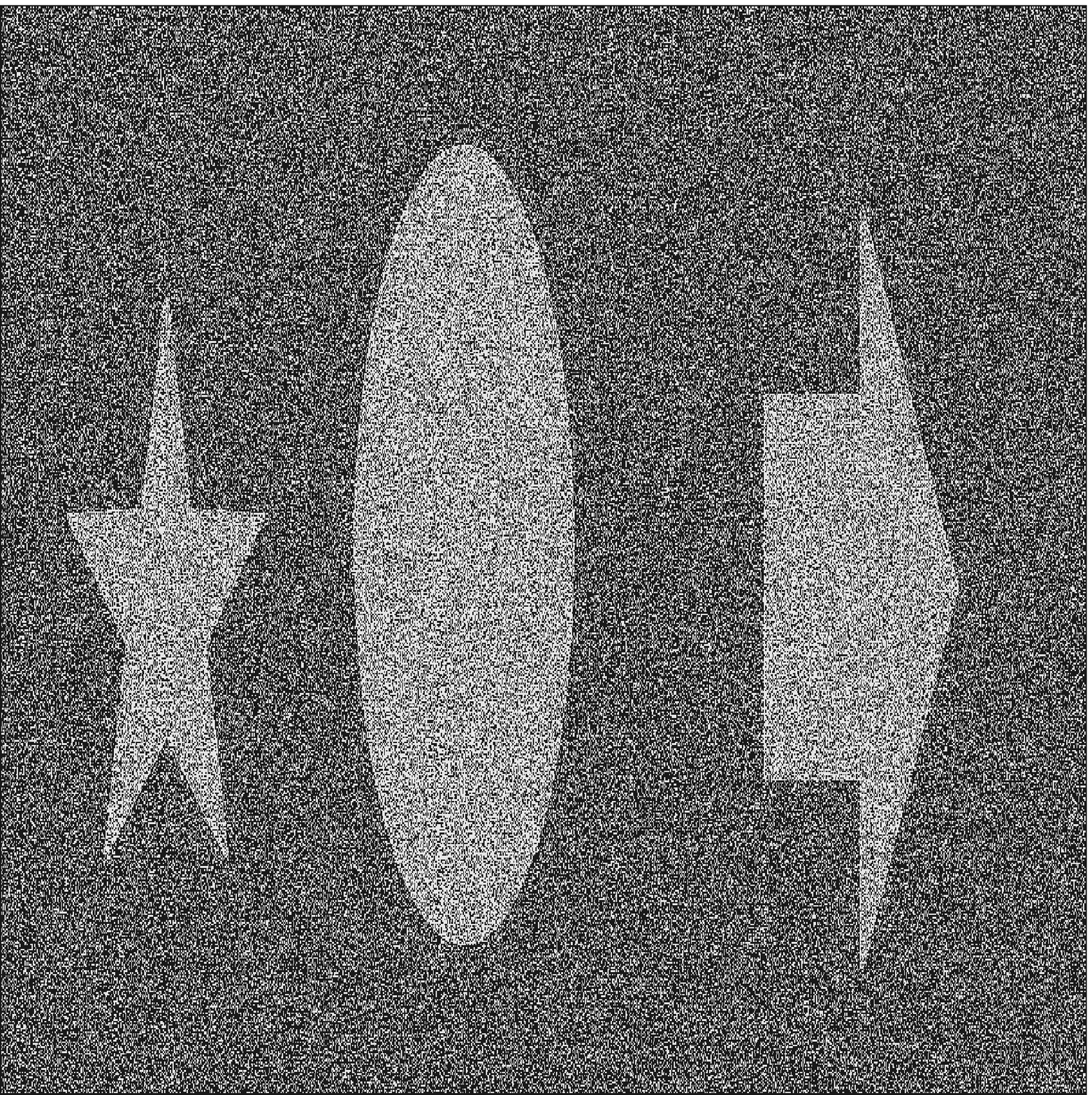}
		}
\subfigure[The initial segmentation of Stage 1]{
			\includegraphics[width=0.45\linewidth, height=0.15\linewidth]{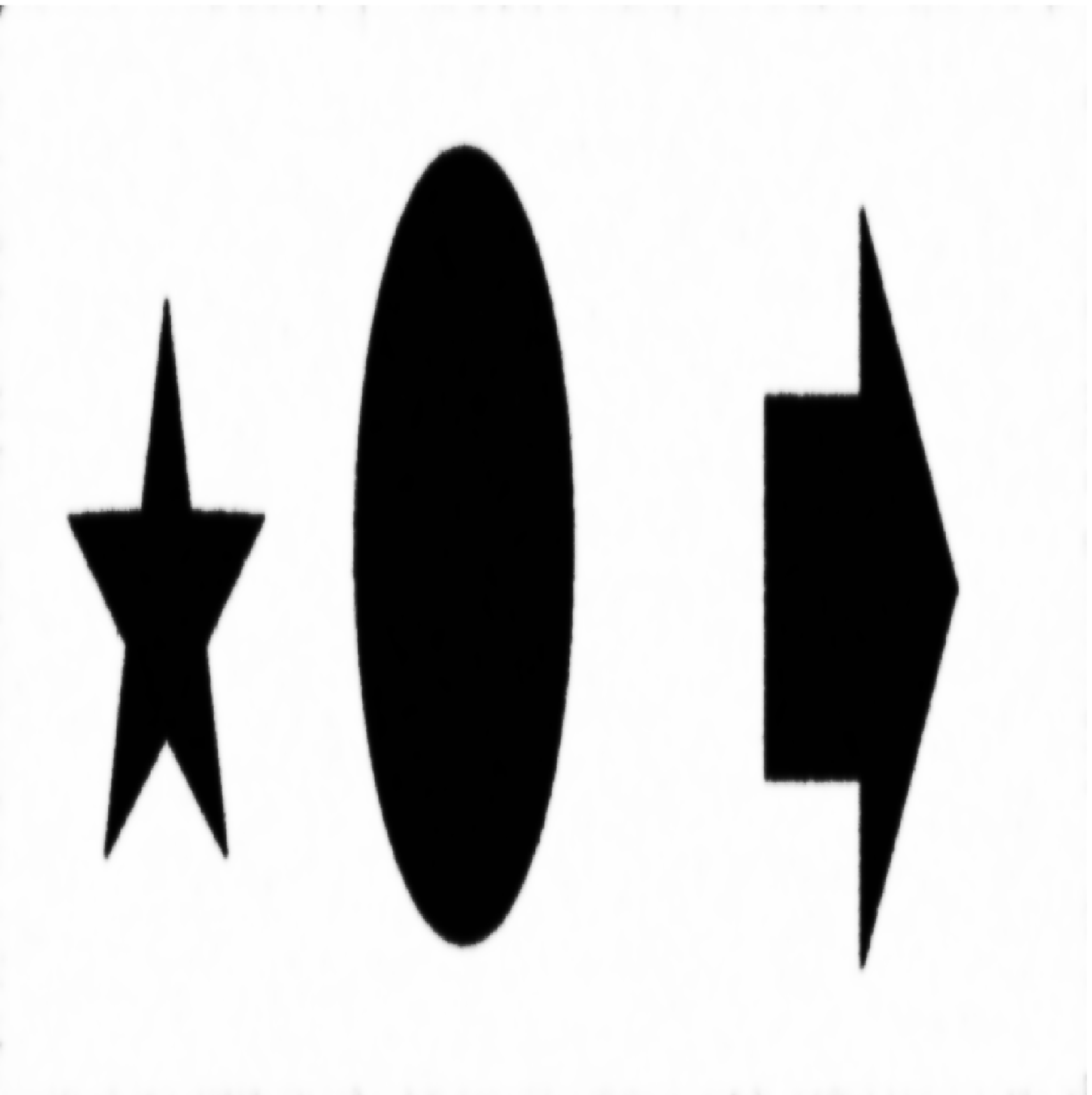}
		}
\subfigure[The final segmentation]{
			\includegraphics[width=0.45\linewidth, height=0.15\linewidth]{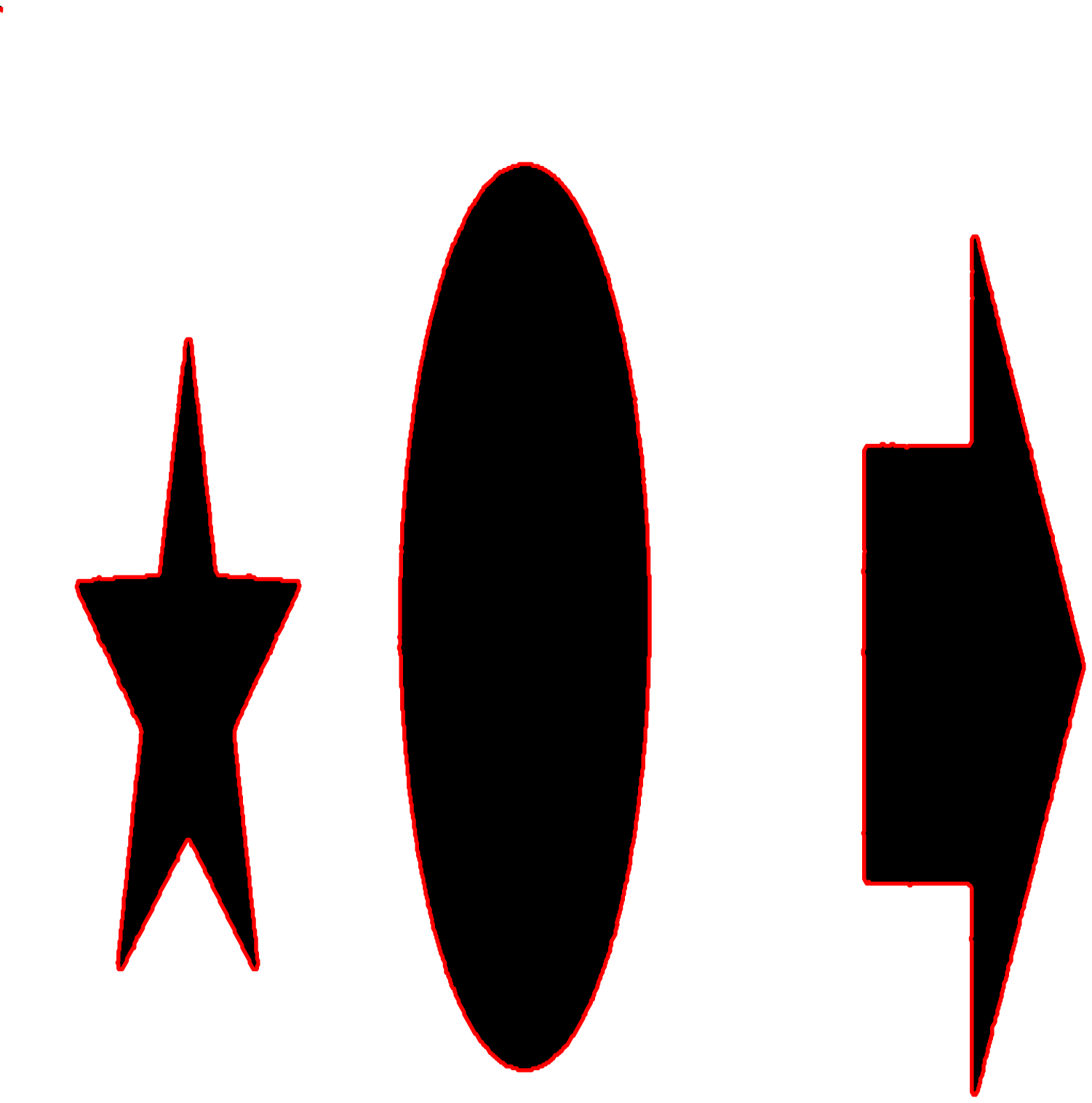}
		}
		\caption{\label{fig:noise} The segmentation of a synthetic image in Example \ref{example_noise}. }
	\end{figure}

\begin{table}[!t]
\resizebox{\textwidth}{!}{
\begin{tabular}{|c|c|c|c|c|c|c|c|c|}
\hline
                 & Stage 1    & \multicolumn{2}{c|}{Stage 2 } & Exterior-loop &  CPU-time & Error & \multicolumn{2}{c|}{Maximum bound}\\ \hline
   Scheme              & $k_1$                  & $k_2$       &$k_i, i>2$          &  $m$  & $T$ & $E$ &min & 1-max\\ \hline
1st-threshold    &  23  &    12      &1         &  6  &  23.1  &5.33E-04 &1.39E-12&1.39E-12\\
2nd-threshold    &  15  &   8         &1         &  3  &  23.8  &5.29E-04 &2.85E-11&2.84E-11 \\
1st-nonlocal    &  36   &  12       &1         &  6 &33.3      &5.67E-04 &8.68E-15&1.52E-14 \\
2nd-nonlocal    &  23   &  8         &1         &  3  & 31.1   & 5.38E-04&1.96E-11&3.10E-11 \\ \hline
\end{tabular}}
\caption{\label{extab:noise} Comparison of four schemes for the syntectic image in  Example \ref{example_noise}.}

\end{table}

\begin{figure}[!t] \centering
		\subfigure[ETD1]{
			\includegraphics[width=0.45\linewidth, height=0.4\linewidth]{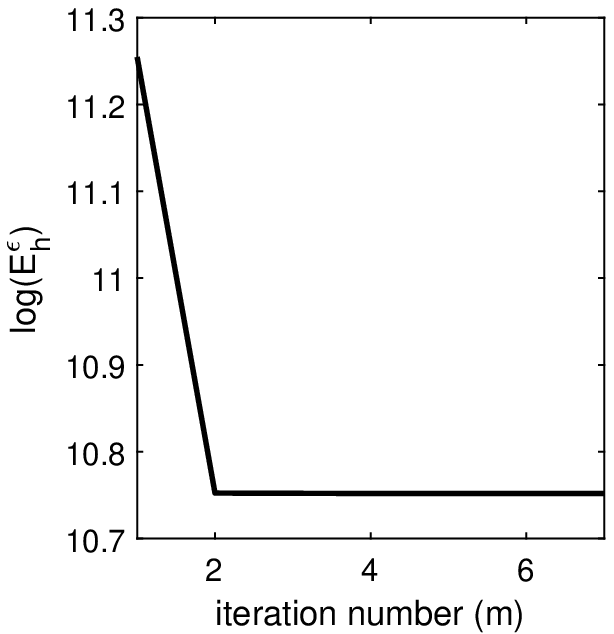}
		}
		\subfigure[ETDRK2 ]{
			\includegraphics[width=0.45\linewidth, height=0.4\linewidth]{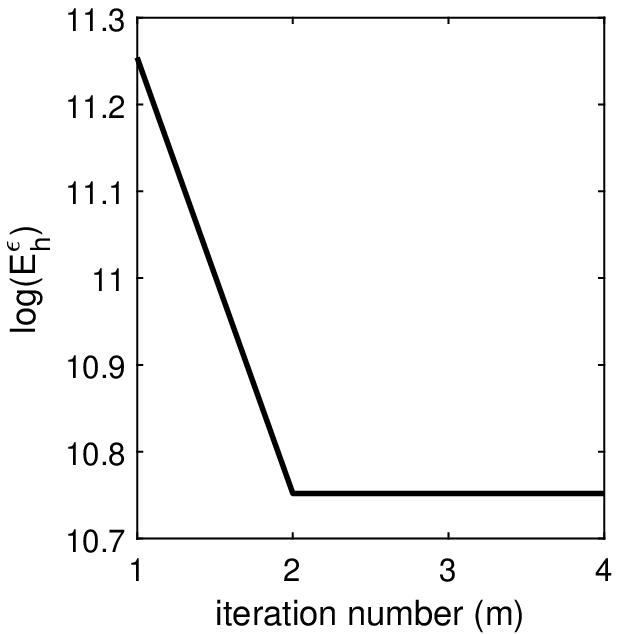}
		}
		\caption{\label{energy_plot} The energy fluctuation of the proposed schemes with the initialization of the nonlocal edge detection method in Example \ref{example_noise}.  }
	\end{figure}

\end{example}

\begin{example}\label{ex:tiger}
This example aims at segmenting the image `tiger in the lake' (size [294,426]). We use two initialization methods to obtain the initial segmentation $u^0$. For the threshold method, we made the threshold value $I_0 = 0.5$. Then in Figure \ref{fig:tiger1}, four plots are listed for solutions from ETD1 and ETDRK2 with two initialization methods, respectively. To underline the segmentation, we use the original image as the background in all four plots. The second-order scheme can avoid the inaccuracy on the neck of the tiger by the first-order scheme. In addition, the nonlocal-detection initialization method can ignore the irrelevant area in the top-right corner without affecting the primary part. And the shape of grass in bottom-right corner is recognized more clearly. As for the most difficult part `the tail', our segmentation is consecutive without any discontinuity. More comparisons of this image `tiger' can be found in \cite{Huang2019_JSC}.  More specifically, some relevant data are stated in Table \ref{tab:tiger_comparison}. The ETDRK2 scheme proceeding with the nonlocal edge detection initialization method can result in better segmentation with less CPU time.  The corresponding parameters are listed as $\lambda_1 = 0.2 ; \lambda_2 = 2.0 $, $\epsilon= 10$ for Stage 1, interaction radius $\delta = 5$ and threshold value $\sigma = 0.05$.
\begin{figure}[t!] \centering
       \subfigure[1st-threshold]{
			\includegraphics[width=0.4\linewidth, height=0.25\linewidth]{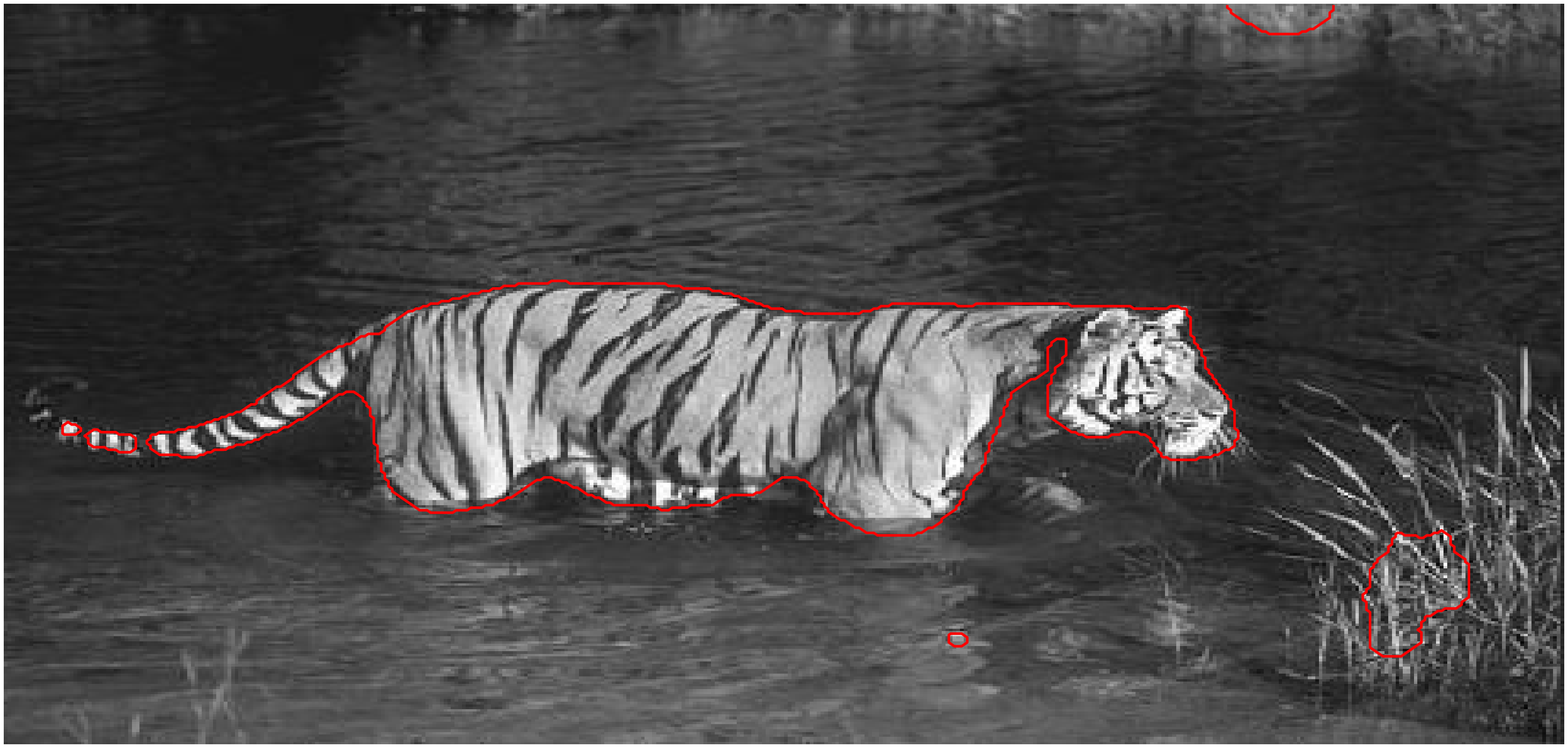}
		         }
        \subfigure[2nd-threshold]{
			\includegraphics[width=0.4\linewidth, height=0.25\linewidth]{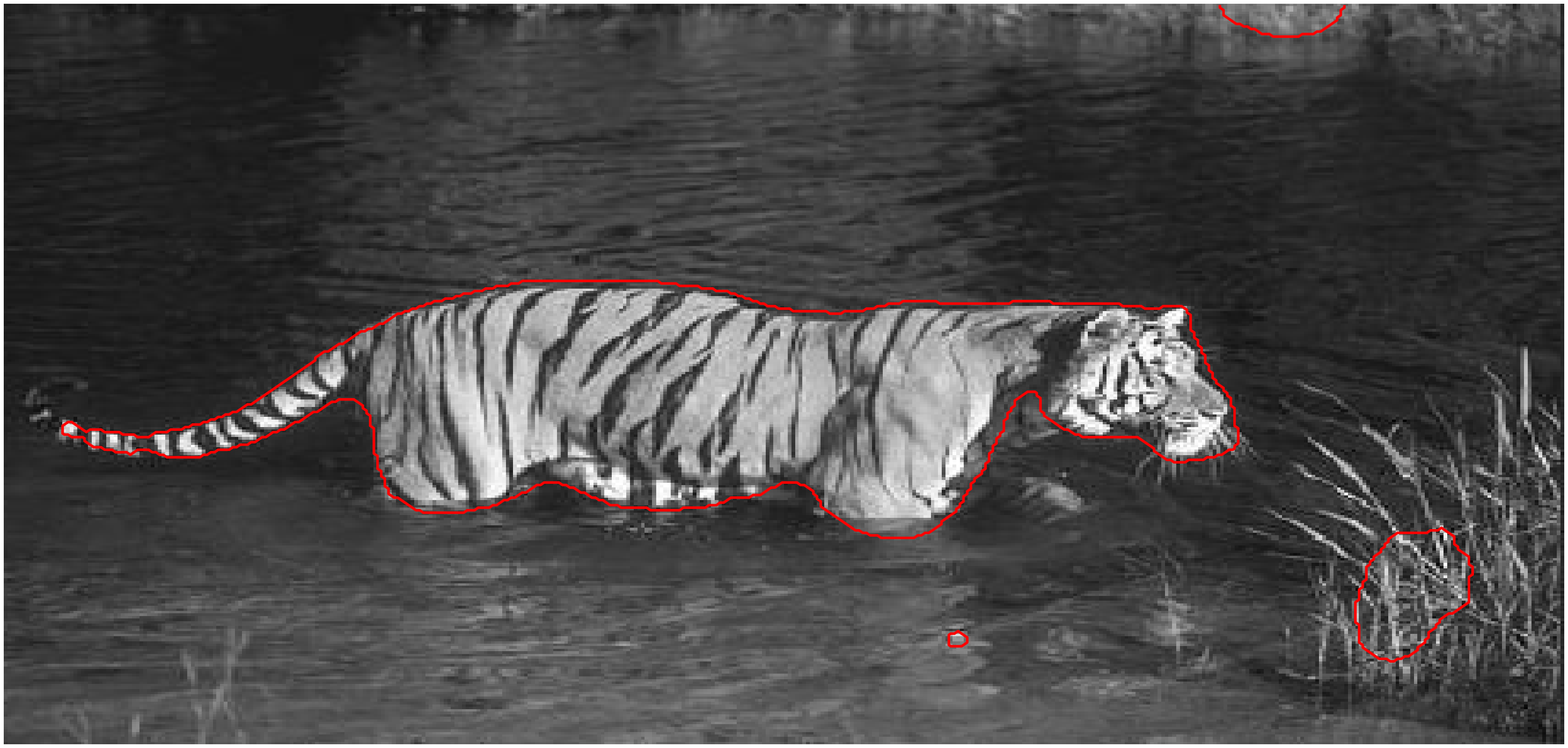}
		}
		\subfigure[1st-nonlocal]{
			\includegraphics[width=0.4\linewidth, height=0.25\linewidth]{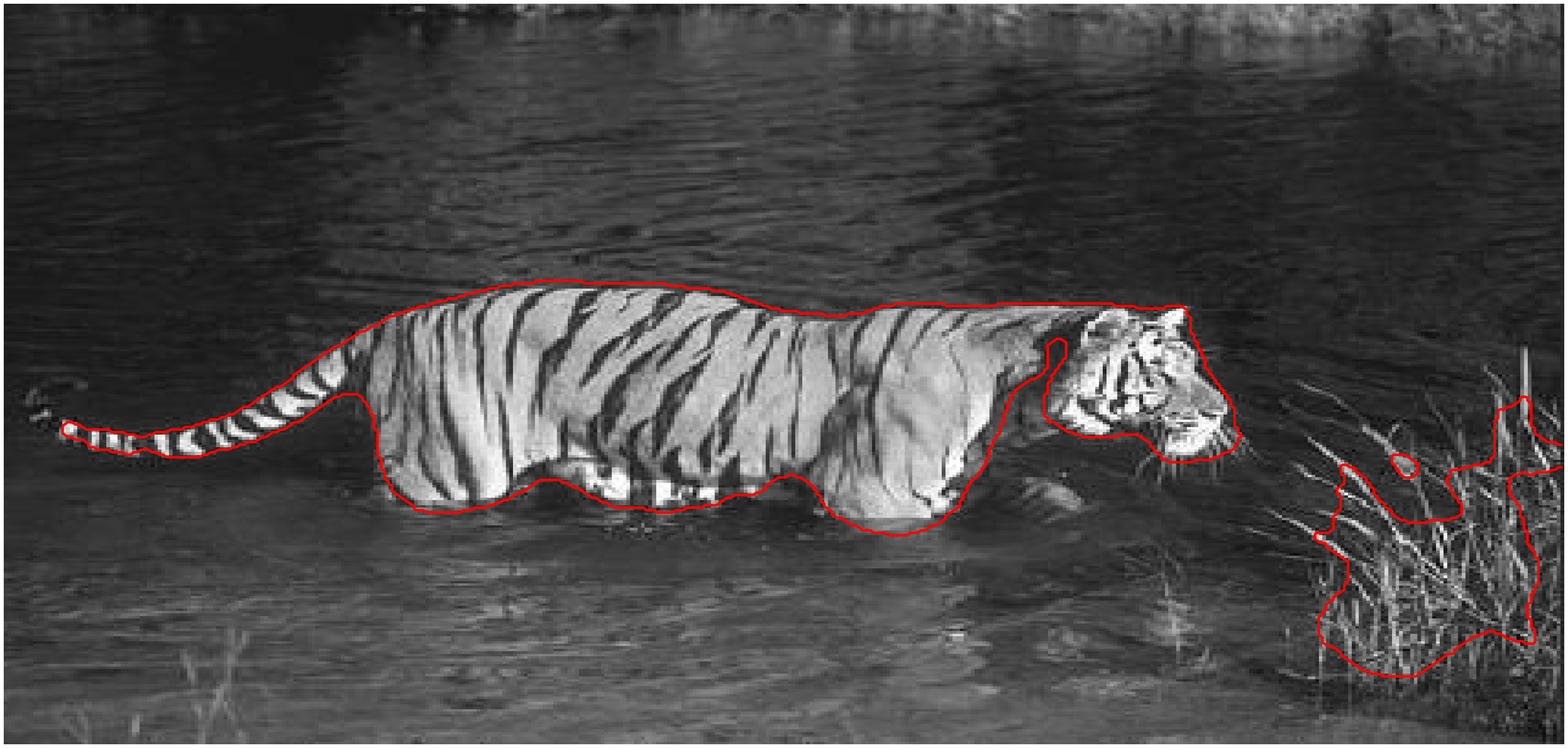}
		}
		\subfigure[2nd-nonlocal]{
			\includegraphics[width=0.4\linewidth, height=0.25\linewidth]{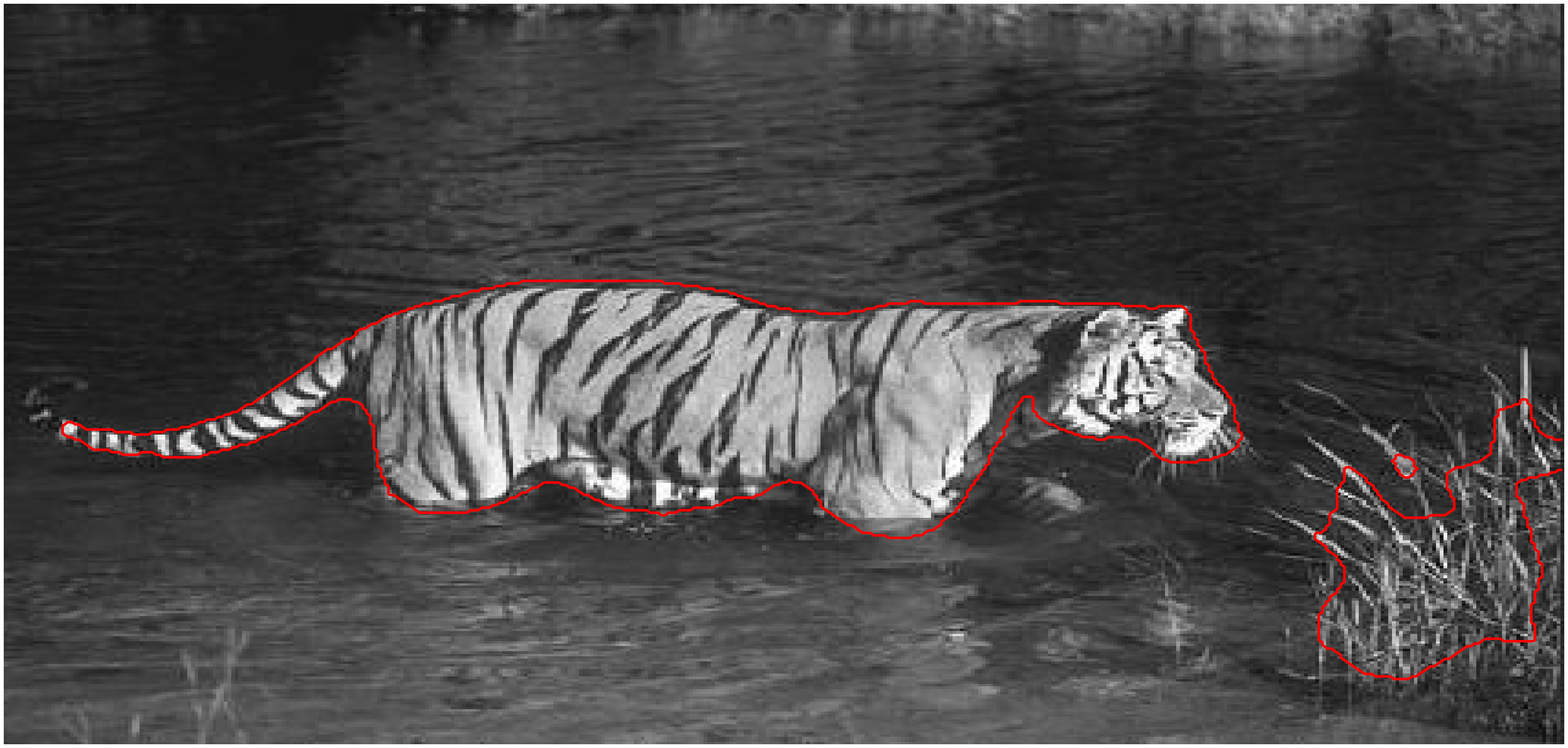}
		}
\caption{\label{fig:tiger1} The segmentation of image `tiger' with two initialization methods in Example \ref{ex:tiger}.}
\end{figure}

\begin{table}[!t]
\resizebox{\textwidth}{!}{
\begin{tabular}{|c|c|c|c|c|c|c|c|}
\hline
                 & Stage 1    & \multicolumn{2}{c|}{Stage 2 } & Exterior-loop &  CPU-time  &\multicolumn{2}{c|}{Maximum bound }\\ \hline
                 Scheme& $k_1$                  & $k_2$       &$k_i, i>2$          &  $m$  &$ T$ &min & 1-max\\ \hline
1st-threshold    &  83  &    16       &1         &  4  &  8.5 &6.07E-14&5.13E-14\\
2nd-threshold    &  65  &   11        &1         &  5  &  10.3   &3.34E-15&4.77E-15 \\
1st-nonlocal     &  83   &  17       &1         &  4 &8.4       &1.53E-14&2.90E-14 \\
2nd-nonlocal     &  64   &  11        &1         &  3  & 9.7   &1.75E-13&3.48E-13\\ \hline
\end{tabular}}
\caption{\label{tab:tiger_comparison} Comparison of four methods for image `tiger' in Example \ref{ex:tiger}.}
\end{table}

\end{example}

\begin{example}\label{ex:mvessel}
This example illustrates the sensitivity of threshold value $I_0$ by image `vessel' (size [344,462]). In Figure \ref{fig:mvessel1}, we use three different threshold values $I_0 = 0.4,0.45,0.5$ to give the initial value $u^0$ of the Allen-Cahn equation at Step 1 of Algorithm 2. We can see that the final segmentation  highly depends on the initialization, i.e., the threshold value $I_0$. When $I_0$ is taken as 0.4, some noise areas are detected, which make the segmentation inaccurate. Conversely, the image is over-smoothing that some small branches of this vessel can not be recognized with $I_0 = 0.5$. Consequently, we narrow down the range of $I_0$ between 0.4 and 0.5, then the final segmentation with $I_0 = 0.45$ is better than the above two. If we shrink the range further, some detailed information should be detected more accurately. However, a bigger $I_0$ will remove the small structures,  while a smaller $I_0$ will show more details but have irrelevant noises. That is, the threshold initial value always exists the less-smoothing  or over-smoothing issue. By contrast, the nonlocal edge detection can capture the detailed texture of the given image. By adjusting the parameters $\delta$ and $\sigma$, more information of the image `vessel' can be detected as an initial value, from which we can get a more delicate segmentation. In Figure \ref{fig:mvessel2}, the nonlocal edge detection method is applied to the initialization, which led to a more specific segmentation.  In Table \ref{tab:mvessel_comparison}, we give the performance comparison of ETD1 and ETDRK2 schemes with threshold initialization and nonlocal edge detection initialization. It is noticed that the ETDRK2 scheme initialized by the nonlocal detection is more efficient than other methods. More comparisons of this image `vessel' can be also found in \cite{Huang2019_JSC}. The remaining parameters are $\lambda_1 = 1 ; \lambda_2 = 7.0 $, $\epsilon= 5 $ for stage 1 , interaction radius $\delta = 5$ and threshold value $\sigma  = 0.01$.
\begin{figure}[!t]
		\subfigure[The original image]{
			\includegraphics[width=0.4\linewidth, height=0.3\linewidth]{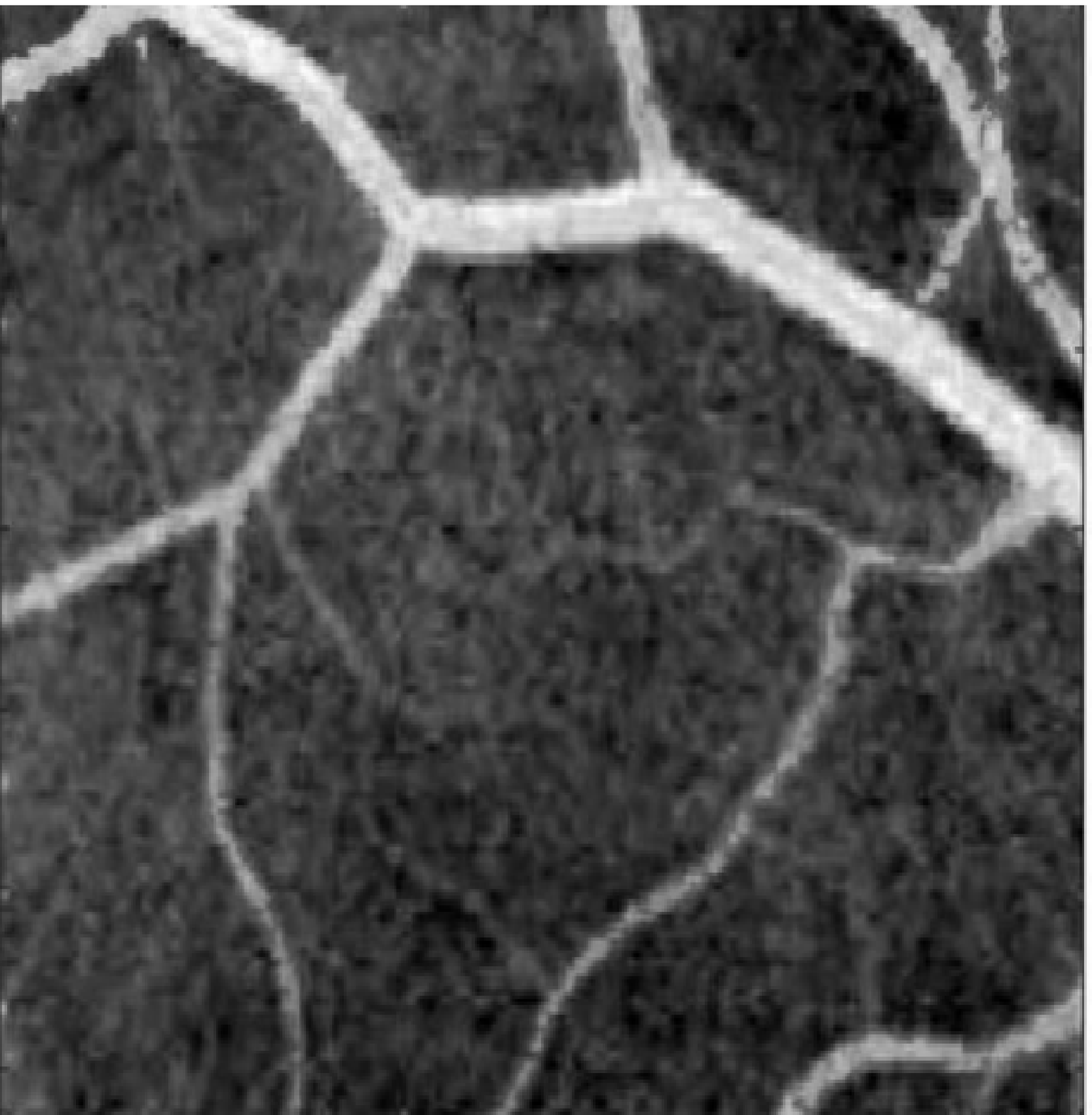}
		} \quad\quad
		\subfigure[Threshold value $I_0=0.4$]{
			\includegraphics[width=0.4\linewidth, height=0.3\linewidth]{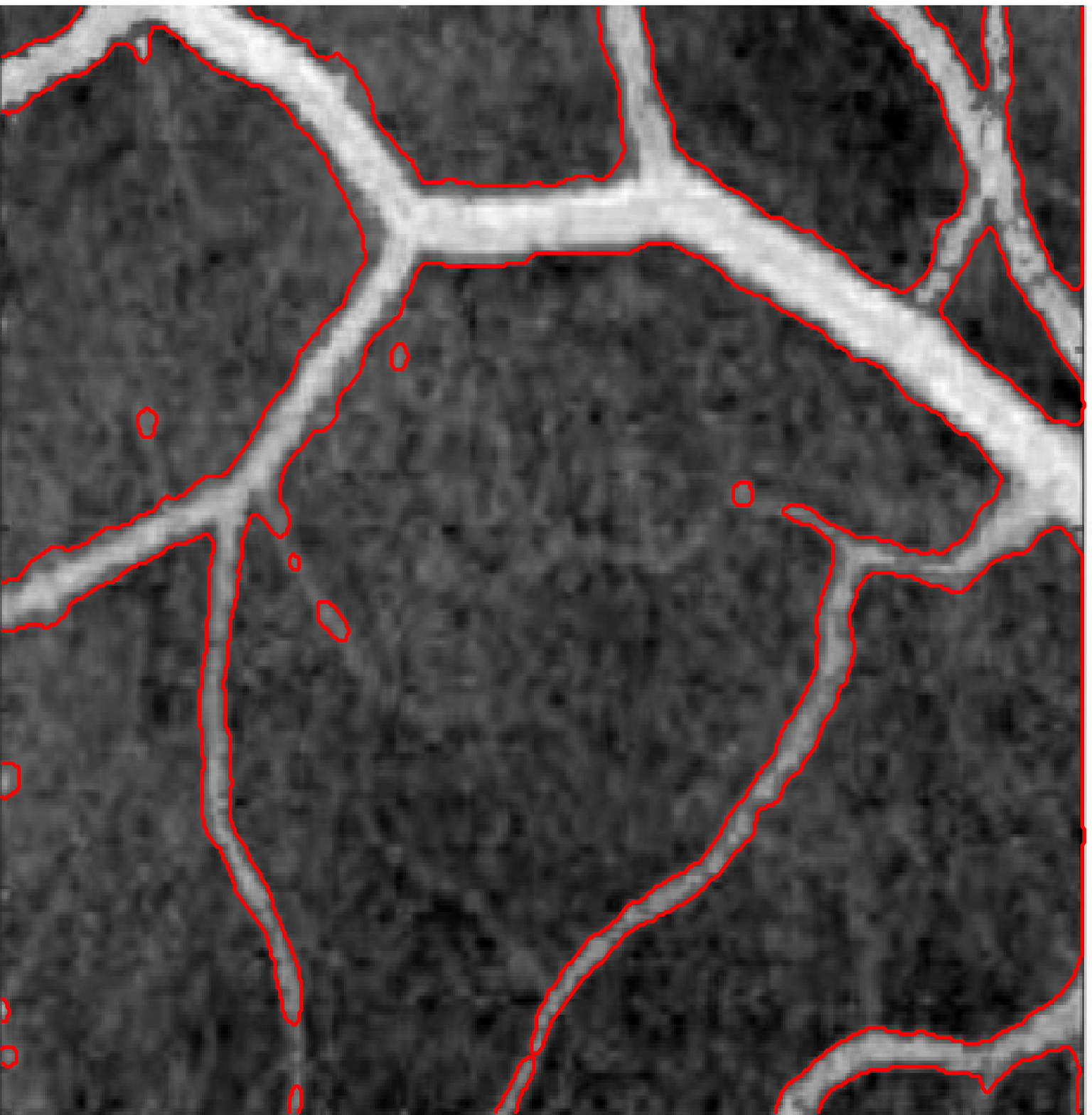}
		}\\
		\subfigure[Threshold value $I_0=0.45$]{
			\includegraphics[width=0.4\linewidth, height=0.3\linewidth]{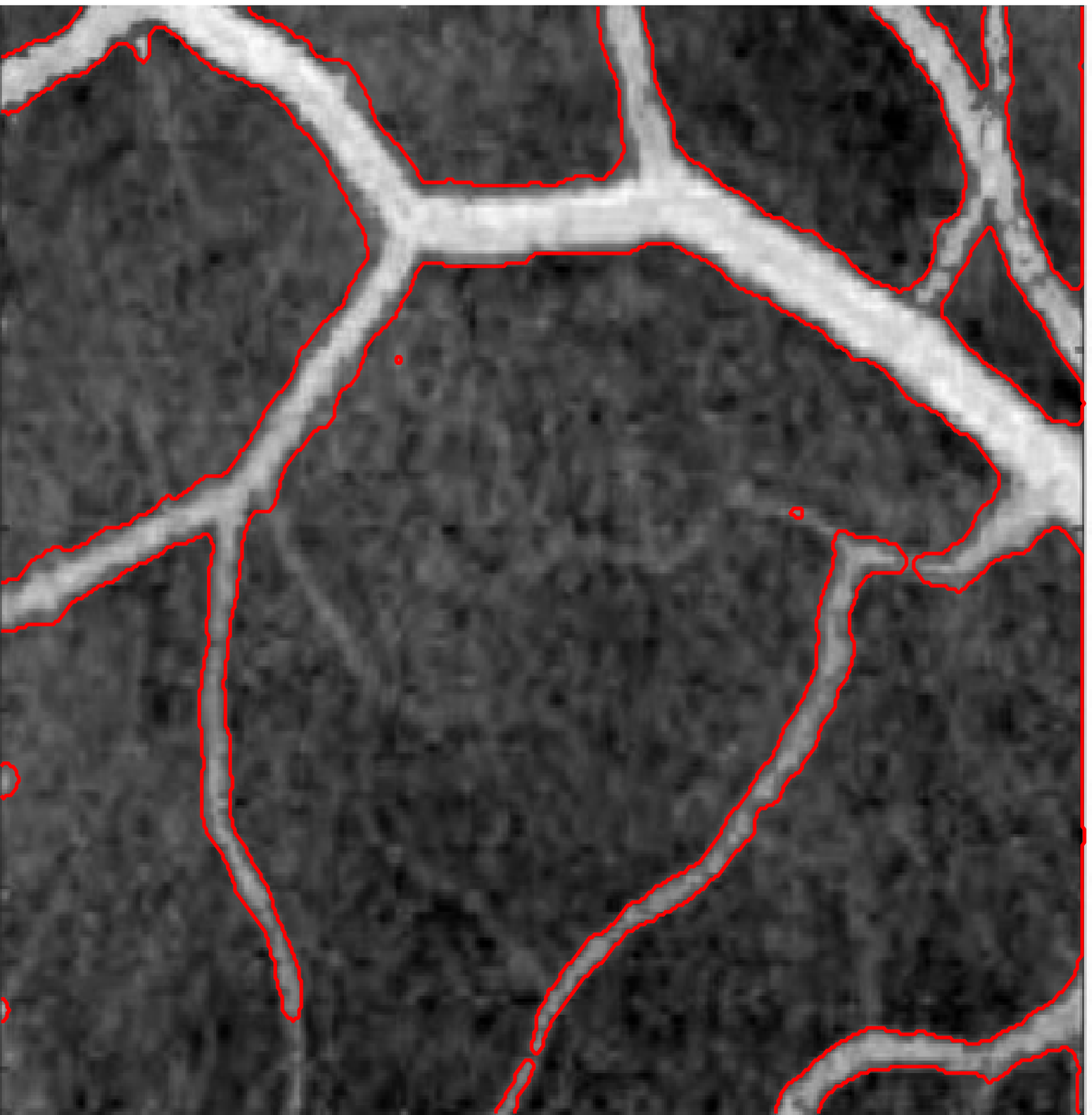}
		} \quad\quad
         \subfigure[Threshold value $I_0=0.5$]{
			\includegraphics[width=0.4\linewidth, height=0.3\linewidth]{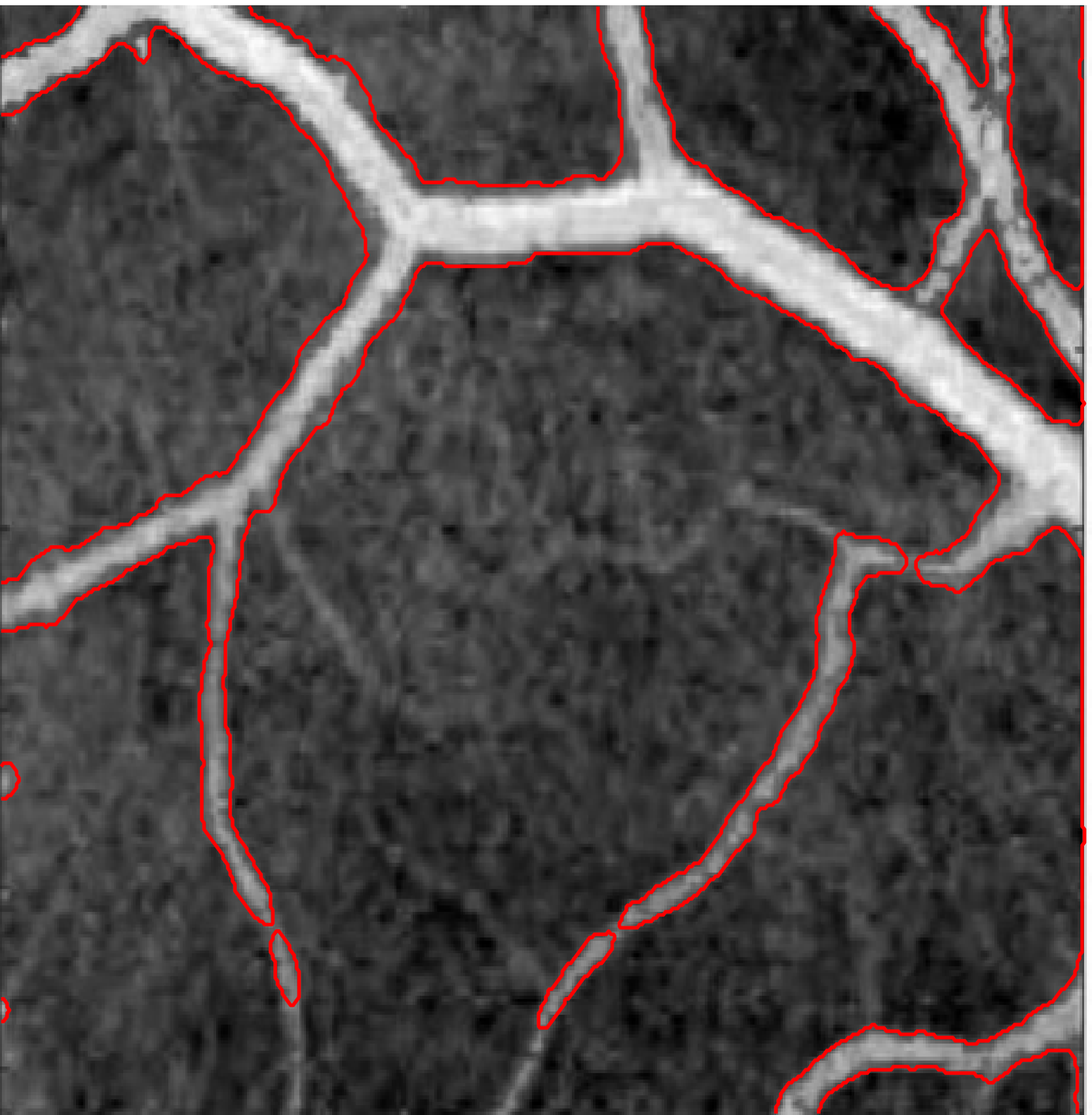}
		}
		\caption{\label{fig:mvessel1} The effect of threshold value $I_0$ on `vessel' segmentation in Example \ref{ex:mvessel}.}
	\end{figure}

\begin{figure}[!t]
		\subfigure[The original image]{
			\includegraphics[width=0.4\linewidth, height=0.3\linewidth]{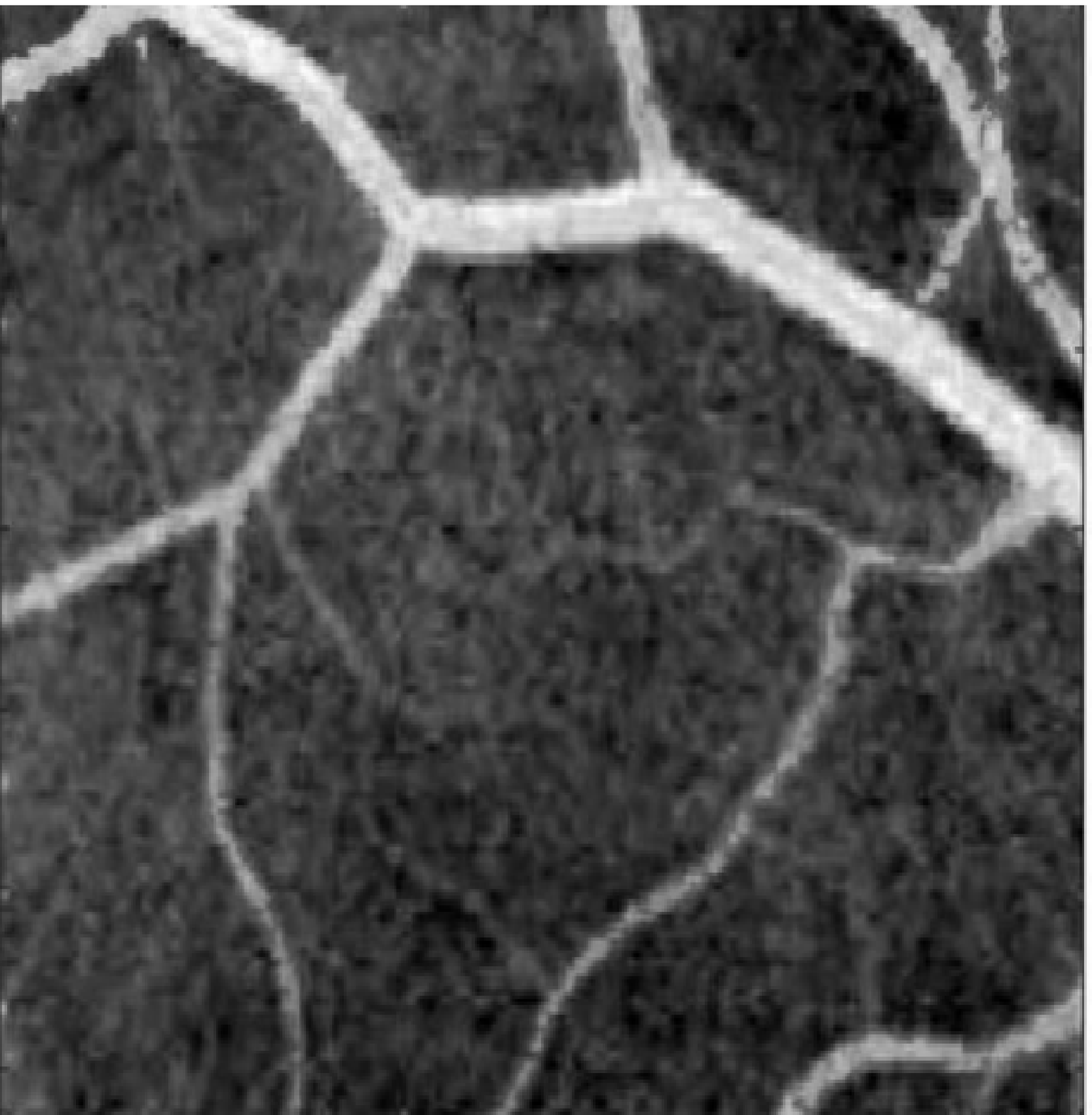}
		} \quad\quad
		\subfigure[The nonlocal edge detection]{
			\includegraphics[width=0.4\linewidth, height=0.3\linewidth]{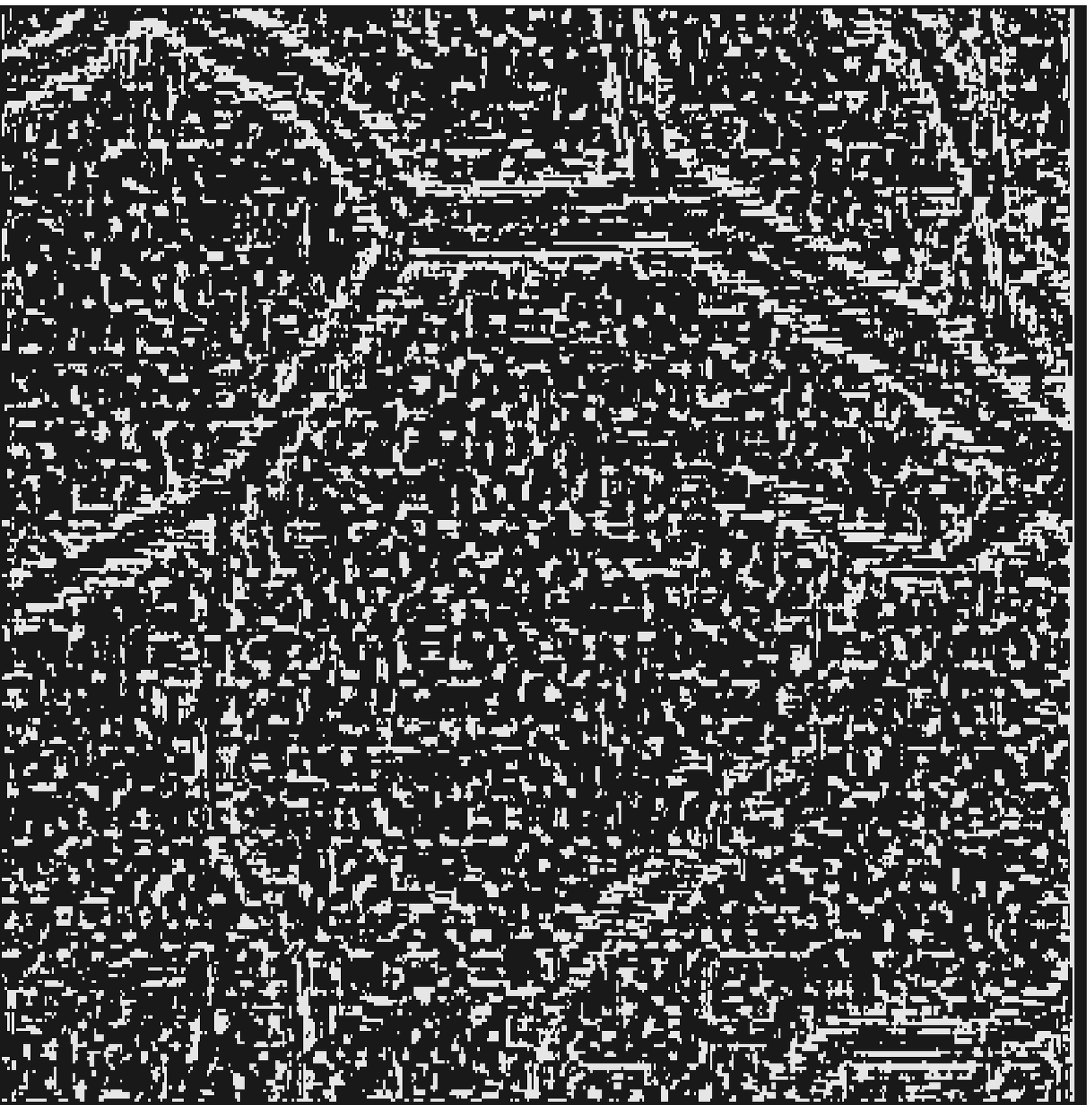}
		}\\
		\subfigure[The result of Stage 1]{
			\includegraphics[width=0.4\linewidth, height=0.3\linewidth]{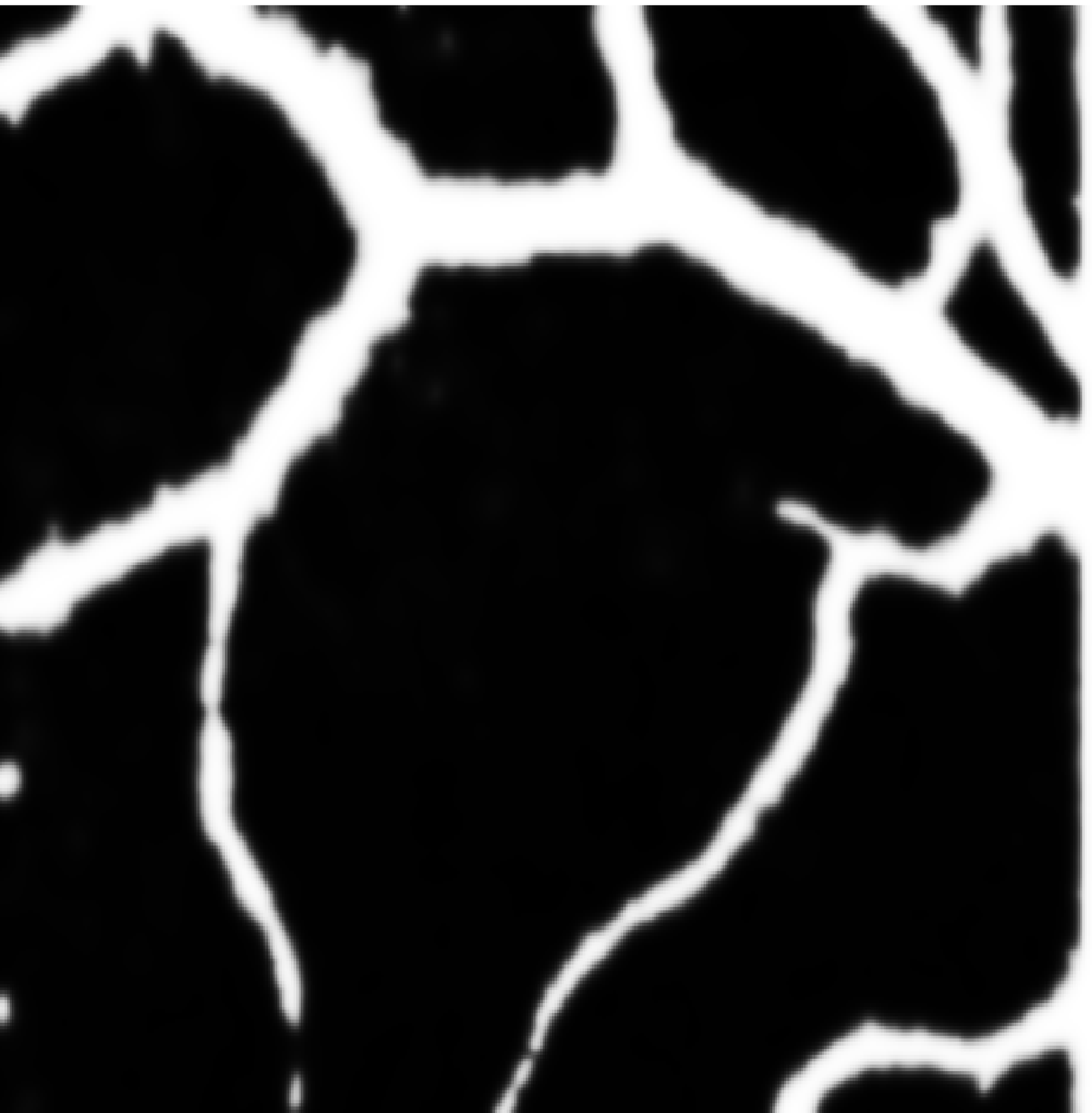}
		}\quad\quad
         \subfigure[The final segmentation]{
			\includegraphics[width=0.4\linewidth, height=0.3\linewidth]{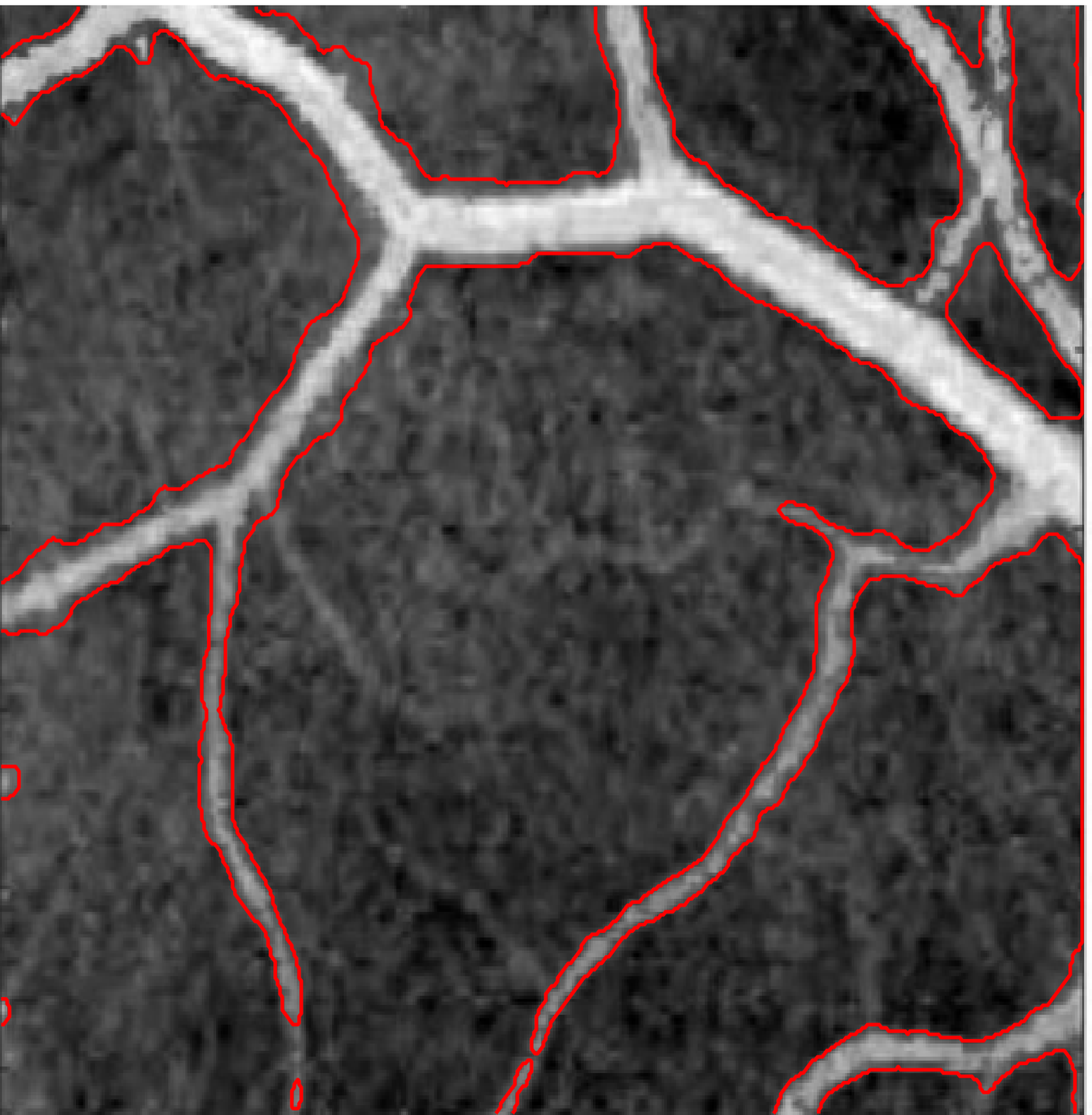}
		}
	\caption{\label{fig:mvessel2} The segmentation of image `vessel' with the initialization of nonlocal edge detection in Example \ref{ex:mvessel}. }
	\end{figure}

\begin{table}[!t]
\resizebox{\textwidth}{!}{
\begin{tabular}{|c|c|c|c|c|c|c|c|}
\hline
                 & Stage 1    & \multicolumn{2}{c|}{Stage 2 } & Exterior-loop &  CPU-time  &\multicolumn{2}{c|}{Maximum bound }\\ \hline
   Scheme              & $k_1$                  & $k_2$       &$k_i, i>2$          &  $m$  & $T$ &min & 1-max\\ \hline
1st-threshold    &  127  &    22       &1         &  13  &  24.0 &7.81E-14&2.95E-14\\
2nd-threshold    &  93 &   13        &1         &  9  &  26.0   &1.60E-14&6.67E-15 \\
1st-nonlocal     &  145   &  20       &1         &  11 &26.7       &1.55E-12&6.09E-13 \\
2nd-nonlocal     &  91   &  13       &1         &  4  & 23.3   &1.28E-11&4.70E-12\\ \hline
\end{tabular}}
\caption{\label{tab:mvessel_comparison} Comparison of four methods for image `vessel' in Example \ref{ex:mvessel}.}
\end{table}

\end{example}

\section{Conclusion}\label{conclusion}
In this paper, we first propose a novel nonlocal edge detection method by using the nonlocal Laplacian operator. In practice, the detection results of this method are similar even superior to some widely used gradient-based edge detection methods. Then to get a better two-phase segmentation of grayscale images, we adopt a phase-field approach of the Chan-Vese model. An alternating minimization algorithm is developed to solve this model. For the derived Allen-Cahn equation, it is solved by exponential time integration and finite difference discretization in space. The nonlocal edge detection gives the initial value for solving the Allen-Cahn equation at the first step. For the developed ETD1 and ETDRK2 schemes, the discrete maximum bound principle and energy stability have been proved theoretically and verified numerically. A variety of numerical experiments have been given to demonstrate the effectiveness of our proposed methods.  It is notable that the initial value provided by the nonlocal edge detection method can generate better segmentations than those from the traditional threshold initialization method. And we could also observe that the ETDRK2 method with the nonlocal edge detection initialization usually gives better results with less CPU time in the simulation. One future direction of this research is to study the segmentation of images with intensity inhomogeneity. Multi-phase image segmentation and segmentation of color images are also of our interests.

\section*{Acknowledgments}
Z. Qiao's work is partially supported by the Hong Kong Research Council RFS grant RFS2021-5S03
and GRF grants 15300417 and 15302919. Q. Zhang's research is supported by the 2019 Hong Kong Scholar Program G-YZ2Y.

\end{document}